\documentclass{amsart}
\usepackage{geometry}
\usepackage{graphicx}
\usepackage{color}
\usepackage{amssymb,amsmath,mathtools}
\usepackage{stmaryrd}

\usepackage{hyperref}

\usepackage{nondim,mydef,boldfonts,xspace}


\begin{document}

\title[Electrowetting]{A Diffuse Interface Model for Electrowetting with Moving Contact Lines}

\author[R.H.~Nochetto]{Ricardo H.~Nochetto$^1$}
\address{$^1$Department of Mathematics and Institute for Physical Science and Technology,
University of Maryland, College Park, MD 20742, USA.}
\email{rhn@math.umd.edu}

\author[A.J.~Salgado]{Abner J.~Salgado$^2$}
\address{$^2$Department of Mathematics, University of Maryland, College Park, MD 20742, USA.}
\email{abnersg@math.umd.edu}

\author[S.W.~Walker]{Shawn W.~Walker$^3$}
\address{$^3$Department of Mathematics and Center for Computation and Technology, Louisiana State University,
Baton Rouge, LA 70803, USA.}
\email{walker@math.lsu.edu}

\thanks{
This material is based on work supported by NSF grants CBET-0754983 and DMS-0807811.
AJS is also supported by an AMS-Simons Travel Grant.
}



\date{Submitted to M3AS on \today}

\begin{abstract}
We introduce a diffuse interface model for the phenomenon of electrowetting on dielectric
and present an analysis of the arising system of equations. Moreover, we study discretization techniques for
the problem. The model takes into account different material parameters on each phase and incorporates the 
most important physical processes, such as incompressibility, electrostatics and dynamic contact lines;
necessary to properly reflect the relevant phenomena.
The arising nonlinear system couples the variable density incompressible Navier-Stokes equations for
velocity and pressure with a Cahn-Hilliard type equation for the phase variable and chemical potential, a
convection diffusion equation for the electric charges and a Poisson equation for the electric potential.
Numerical experiments are presented, which illustrate the wide range of
effects the model is able to capture, such as splitting and coalescence of droplets.
\end{abstract}

\maketitle

\section{Introduction}
\label{sec:Intro}
The term electrowetting on dielectric refers to the local modification of the surface tension between 
two immiscible fluids via electric actuation. This allows for change of shape and wetting behavior 
of a the two-fluid system and, thus, for its manipulation and control.

The existence of such a phenomenon was originally discovered by Lippmann \cite{Lippmann1875},
more than a century ago (see also 
\cite{PhysRevB.76.035437, MugelePhysics, Berge:Comptes_Rendus, Shapiro:JAP}). However, only recently has 
electrowetting found a wide spectrum of applications, specially in the realm of micro-fluidics 
\cite{Cho_Moon:ME_Congress_2001, Cho:JMEMS, Gong:MEMSConf}. One can mention, for example, reprogrammable 
lab-on-chip systems \cite{Lee:Sens_Act_2002, SaekiKim_PMSE2001}, auto-focus cell phone lenses 
\cite{BergePeseux_EPJ2000}, colored oil pixels and video speed smart paper
\cite{HayesFeenstra_Nature2003, Roques-CarmesFeenstra_JAP2004, Roques-CarmesHayes_JAP2004}. In 
\cite{NatureInverseEwod}, the reverse electrowetting process has been proposed as an approach to energy 
harvesting.

From the examples presented above, it becomes clear that it is very important for applications to 
have a better understanding of this phenomenon and it is necessary to obtain reliable computational 
tools for the simulation and control of these effects. The computational models must be complete enough, 
so that they can reproduce
the most important physical effects, yet sufficiently simple that it is possible to extract from them
meaningful information in a reasonable amount of computing time.
Several works have been concerned with the modeling of electrowetting. The approaches include
experimental relations and scaling laws \cite{ewodExperimentJapanese,springerlink:10.1007/s10404-008-0360-y},
empirical models \cite{CambridgeJournals:1380872},
studies concerning the dependence of the contact angle (\cite{MR2551398,MR2483669})
or the shape of the droplet (\cite{MR2487069,MR2683577}) on the applied voltage,
lattice Boltzmann methods \cite{MR2745030,MR2557502} and others.
Of relevance to our present discussion are the
works \cite{walker:102103,MR2595379} and \cite{MR2511642,fontelos:527}. To the best of our knowledge,
\cite{walker:102103,MR2595379} are the first papers where the contact line pinning was included in
an electrowetting model. On the other hand the models of \cite{MR2511642,fontelos:527} are the only ones
that are intrinsically three dimensional and do not assume any special geometric configuration. They
have the limitation, however, that they assume the density of the two fluids to be constant and they apply
a no-slip boundary condition to the fluid-solid interface, thus limiting the movement of the droplet.

The purpose of this work is to propose and analyze an electrowetting model that is intrinsically
three-dimensional; it takes into account that all material parameters are different in each one of 
the fluids; and it is derived (as long as this is possible) from physical principles. To do so, 
we extend the diffuse
interface model of \cite{MR2511642}. The main additions are the fact that we allow the fluids to
have different densities -- thus leading to a variable density Cahn Hilliard Navier Stokes system -- 
and that we treat the contact line movement in a thermodynamically consistent
way, namely using the so-called generalized Navier boundary condition (see \cite{MR2261865,QianCiCP}).
In addition, we propose a (phenomenological) approach to contact line pinning and study stability and 
convergence of discretization techniques. In this respect, our work also differs from 
\cite{MR2511642,fontelos:527}, since our approach deals with a practical fully discrete scheme, for
which we derive a priori estimates and convergence results.

Through private communication we have become aware of the following recent contributions: discretization
schemes for the model proposed in \cite{MR2511642} are studied in \cite{klingbeilpaja}; the models of
\cite{MR2511642,fontelos:527} have been extended, using the techniques of \cite{AbelsGarckeGrun}, in
\cite{Grunmodelpaja,grunklingbeilpaja} where discretization issues are also discussed.

This work is organized as follows. In \S\ref{sub:Notation} we introduce the notation and some preliminary
assumptions necessary for our discussion. Section~\ref{sec:Model} describes the model that we shall be
concerned with and its physical derivation. A formal energy estimate and a formal weak formulation of our
problem is shown in section~\ref{sec:Energy}. The energy estimate shown in this section serves as
a basis for the precise definition of our notion of solution and the proof of its existence. The details of
this are accounted for in section~\ref{sec:Discrete}. In section~\ref{sec:NumExp} we discuss
discretization techniques for our problem and present some numerical experiments aimed at showing the
capabilities of our model: droplet splitting and coalescence as well as contact line movement.
Finally, in section~\ref{sec:semidiscrete}, we briefly discuss convergence of the discrete solutions
to solutions of a semi-discrete problem.

\subsection{Notation and Preliminaries}
\label{sub:Notation}
\begin{figure}[h]
\label{fig:conf}
\includegraphics[scale=0.5]{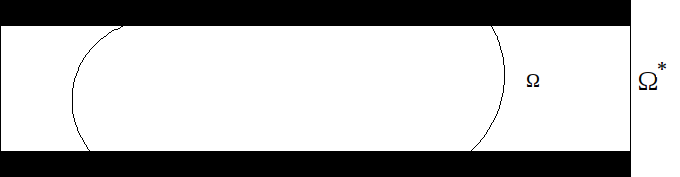}
\caption{The basic configuration of an electrowetting on dielectric device 
\cite{Cho_Moon:ME_Congress_2001, Cho:JMEMS}. The solid black region
depicts the dielectric plates and the white region denotes a droplet of one fluid (say water),
which is surrounded by another (air). We denote by $\Omega$ the fluid domain,
by $\Gamma$ its boundary, by $\Omega^\star$ the region occupied by the fluids and the plates
and by $\partial^\star\Omega^\star := \partial\Omega^\star \setminus \Gamma$.}
\end{figure}

Figure~\ref{fig:conf} shows the basic configuration for the
electrowetting on dielectric problem. We use the symbol $\Omega$ to denote the
domain occupied by the fluid and $\Omega^\star$ for the fluid and dielectric plates, thus,
$\Omega \subset \Omega^\star$.
In this manner, we assume that $\Omega$ and $\Omega^\star$ are convex, bounded connected
domains in $\Real^d$, for $d=2$ or $3$, with $\calC^{0,1}$ boundaries. The boundary of $\Omega$ is
denoted by $\Gamma$ and $\partial^\star \Omega^\star = \partial\Omega^\star \setminus \Gamma$,
$\bn$ stands for the outer unit normal to $\Gamma$.
We denote by
$[0,T]$ with $0<T<\infty$ the time interval of interest.
For any vector valued function $\bw :\Omega \rightarrow \Real^d$ that is smooth enough so as to have a trace on
$\Gamma$, we define the tangential component of $\bw$ as
\begin{equation}
  \bw_{\btau}|_\Gamma := \bw|_\Gamma - (\bw|_\Gamma \SCAL \bn ) \bn,
\label{eq:defttrace}
\end{equation}
and, for any scalar function $f$, $\partial_\btau f := (\GRAD f)_\btau$.

We will use standard notation for spaces of Lebesgue integrable functions $L^p(\Omega),\ 1\leq p \leq \infty$
and Sobolev spaces $W^m_p(\Omega)\ 1\leq p \leq \infty,\ m \in \polN_0$, \cite{02208228}.
Vector valued functions and spaces of vector valued functions will be denoted by boldface characters.
For $S\subset \Real^d$, by $\langle \cdot, \cdot\rangle_S$ we denote, indistinctly, the $L^2(S)$-
or $\bL^2(S)$-inner product. If no subscript is given, we assume that the domain is $\Omega$.
If $S \subset \Real^{d-1}$, then the inner product is denoted by $[\cdot,\cdot]_S$ and if no subindex is
given, the domain must be understood to be $\Gamma$.
We define the following spaces:
\begin{equation}
  \Hunstar := \left\{ v \in H^1(\Omega^\star): v|_{\partial^\star\Omega^\star}=0 \right\},
\label{eq:defHunstar}
\end{equation}
normed by
\[
  \| v \|_{H^1_\star} := \| \GRAD v \|_{\bL^2(\Omega^\star)},
\]
and
\begin{equation}
  \bV := \left\{ \bv \in \Hund: \bv\SCAL\bn|_\Gamma =0   \right\},
\label{eq:defbV}
\end{equation}
which we endow with the norm
\[
  \| \bv \|_{\bV}^2 := \| \GRAD \bv \|_{\bL^2}^2 + \| \bv_\btau \|_{\bL^2(\Gamma)}^2.
\]
Clearly, for these norms, they are Hilbert spaces.

To take into account the fact that our problem will be time dependent we introduce the following notation.
Let $E$ be a normed space with
norm $\|\cdot\|_E$. The space of functions $\varphi : [0,T] \rightarrow E$ such that the map
$(0,T)\ni t \mapsto \| \varphi(t)\|_E \in \Real$ is $L^p$-integrable is denoted by
$L^p(0,T,E)$ or $L^p(E)$. To discuss the time discretization of our problem, we introduce a time-step
$\dt>0$ (for simplicity assumed constant) and let $t_n = n \dt$ for $0 \leq n \leq N:=\lceil T/\dt \rceil$.
for any time-dependent function, $\varphi$, we denote $\varphi^n := \varphi(t_n)$ and the sequence of values
$\{ \varphi^n \}_{n=0}^N$ is denoted by $\varphi_\dt$. For any sequence $\varphi_\dt$ we define the
time-increment operator $\frakd$ by
\begin{equation}
  \frakd \varphi^n := \varphi^n - \varphi^{n-1},
\label{eq:frakd}
\end{equation}
and the time average operator $\overline{(\cdot)}$ by
\begin{equation}
  \overline{\varphi^n} := \frac12\left( \varphi^n + \varphi^{n-1} \right).
\label{eq:defstar}
\end{equation}
On sequences $\varphi_\dt \subset E$ we define the norms
\[
  \| \varphi_\dt \|^2_{\ell^2(E)} := \dt \sum_{n=0}^N \| \varphi^n\|_E^2,
  \quad
  \| \varphi_\dt \|_{\ell^\infty(E)} := \max_{0 \leq n \leq N} \left\{ \| \varphi^n \|_E  \right\},
  \quad
  \| \varphi_\dt \|^2_{\frakh^{1/2}(E)} := \sum_{n=1}^N \| \frakd \varphi^n \|_E^2.
\]
which are, respectively, discrete analogues of the $L^2(E)$, $L^\infty(E)$ and $H^{1/2}(E)$ norms. When
dealing with energy estimates of time discrete problems, we will make, without explicit mention,
repeated use of the following elementary identity
\begin{equation}
  2a(a-b) = a^2 - b^2 + (a-b)^2.
\label{eq:toobasic}
\end{equation}

\section{Model Derivation}
\label{sec:Model}
In this section we briefly describe the derivation of our model. The procedure
used to obtain it is quite similar to the arguments used in \cite{MR2511642,MR2261865,AbelsGarckeGrun}
and
it fits into the general ideological framework of so-called phase-field models.
In phase-field methods, sharp interfaces are replaced by thin transitional layers where the interfacial
forces are now smoothly distributed and, thus, there is no need to explicitly track interfaces.

\subsection{Diffuse Interface Model}
\label{sub:Model}
To develop a phase-field model, we begin by introducing a so-called phase field variable $\phi$ and
an interface thickness $\delta$. The phase field variable acts as a marker that will be almost constant
(in our case $\pm1$) in the bulk regions, and will smoothly transition between these values in an
interfacial region of thickness $\delta$. Having introduced the phase field,
all the material properties that depend on the phase are slave variables and defined as
\begin{equation}
  \Psi(\phi) = \frac{ \Psi_1 - \Psi_2 }2 \arctan\left(\frac\phi\delta\right) + \frac{ \Psi_1 + \Psi_2 }2,
\label{eq:slave}
\end{equation}
where the $\Psi_i$ are the values on each one of the phases.

\begin{rem}[Material properties]
Relation \eqref{eq:slave} is not the only possible definition of the phase dependent quantities. For
instance, \cite{Shen_Xiaofeng_2010} proposes to use a linear average between the bulk values. This approach has
the advantage that the derivative of a phase-dependent field with respect to the phase (expressions
that contain such quantities appear repeatedly) is constant, which greatly simplifies the calculations.
However, this definition cannot be guaranteed to stay in the physical range of values which might lead to,
say, a vanishing density or viscosity. On the other hand, \cite{MR1984386} proposes to use a harmonic average
which guarantees that positive quantities stay bounded away from zero. In this work, we will assume that,
with the exception of the permittivity $\vare$,
\eqref{eq:slave} is the way the slave variables are defined, which has the advantage that guarantees that
the field stays within the physical bounds. Any other definition with this property is equally suitable
for our purposes.
\end{rem}

We model the droplet and surrounding medium as an incompressible Newtonian viscous two-phase fluid,
so that its behavior is governed by the variable density incompressible Navier Stokes equations.
The equation of conservation of momentum can be written in several forms. We chose the one proposed
by Guermond and Quartapelle (\cite{MR2002i:76080}, see also
\cite{MR2726065,Shen_Xiaofeng_2010}) because its nonlinear term possesses a skew symmetry property
similar to the constant density Navier Stokes equations,
\begin{subequations}\label{eq:NSE}
\begin{align}
\label{eq:NSE1}
  \sigma(\sigma\ue)_t + \left( \rho\ue\SCAL\GRAD\ue \right) + \frac12 \DIV(\rho\ue )\ue
  - \DIV\left(\eta \bS(\ue) \right) + \GRAD \pe &= \bF, \\
  \DIV \ue &= 0,
\label{eq:NSE2}
\end{align}
\end{subequations}
where $\sigma = \sqrt\rho$ and $\rho$ is the density of the fluid and depends on the phase field;
$\ue$ is the velocity of the fluid;
$\pe$ is the pressure; $\eta$ is the viscosity of the fluid and depends on $\phi$;
$\bS(\ue) = \tfrac12( \GRAD \ue + \GRAD \ue^\intercal )$ 
is the symmetric part of the gradient
and $\bF$ are the external forces acting on the fluid.

The phase field can be thought of as a scalar that is convected by the flow. Hence its motion
is described by
\begin{equation}
  \phi_t + \DIV( \phi \ue ) = -\DIV\bJ_\phi,
\label{eq:phi}
\end{equation}
for some flux field $\bJ_\phi$ which will be found later.

To model the interaction between the applied voltage and the fluid we introduce the charge density $q$.
Another possibility, not explored here, is to introduce ion concentrations, thus leading to a Nernst Planck
Poisson-like system,
see \cite{fontelos:527,MR2535842,MR2666654,MR2471611}.
The electric displacement field $\bD$ is defined in $\Omega^\star$. The evolution of these two quantities is
governed by Maxwell's equations, \ie
\begin{equation}
 \DIV\bD = q, \qquad \bD_t + q\ue + \bJ_\bD =0,
\label{eq:DMaxwell}
\end{equation}
for some flux $\bJ_\bD$. Notice that we assume that the magnitude of the velocity of the fluid is 
negligible in comparison with the speed of light, and that the frequency of voltage actuation 
is sufficiently small, so that magnetic effects can be ignored. Taking the time derivative of the
first equation and substituting in the second we obtain
\begin{equation}
  q_t + \DIV(q\ue) = -\DIV \bJ_\bD.
\label{eq:qD}
\end{equation}

To close the system, we must prescribe boundary conditions, determine the force $\bF$ exerted
on the fluid, and find constitutive relations for the fluxes $\bJ_\phi$ and $\bJ_\bD$. We are
assuming the solid walls are impermeable, therefore if $\bn$ is the normal to $\Gamma$,
$\ue\SCAL\bn = 0$ on $\Gamma$ and $\bJ_\flat \SCAL \bn = 0$ for any flux $\bJ_\flat$. To find the
rest of the boundary conditions, $\bF$ and relations for the fluxes, we denote the surface tension
between the two phases by $\gamma$ and define the Ginzburg-Landau double well potential by
\[
  \calW(\xi) = \begin{dcases}
                  (\xi + 1 )^2, & \xi < -1, \\
                  \frac14\left( 1 - \xi^2 \right)^2, & |\xi| \leq 1, \\
                  (\xi - 1 )^2, & \xi > 1.
                \end{dcases}
\]
\begin{rem}[The Ginzburg Landau potential]
The original definition, given by Cahn and Hilliard, of the potential is logarithmic. See, for instance,
\cite{Gomez20115310}. This way, the potential becomes infinite if the phase field variable is out of
the range $[-1,1]$, thus guaranteeing that the phase field variable $\phi$ stays within that range.
This is difficult to treat both in the analysis and numerics and hence
practitioners have used the Ginzburg-Landau potential $c(1-\xi^2)^2$, for some $c>0$. We go one step
further and restrict the growth of the potential to quadratic away from the range of interest. With this
restriction Caffarelli and M\"uller, \cite{MR1367359}, have shown uniform $L^\infty$-bounds on the solutions
of the Cahn Hilliard equations (which as we will see below the phase field must satisfy). This has also
proved useful in the numerical discretization of the Cahn Hilliard and Cahn Hilliard Navier
Stokes equations, see \cite{MR2679727,MR2726065,SalgadoMCL}.
\end{rem}

Finally, we introduce the interface energy density function, which describes the energy due to the 
fluid-solid interaction.
Let $\theta_s$ be the contact angle that, at equilibrium, the interface between the two fluids makes with
respect to the solid walls (see \cite{MR2261865,MR2498521,SalgadoMCL}) and define
\[
  \Theta_{fs}(\phi) = \frac{\cos\theta_s}2 \sin\left( \frac{\pi\phi}2 \right).
\]
Then, up to a constant, the interfacial energy density equals $\gamma \Theta_{fs}(\phi)$.

Let us write the free energy of the system
\begin{equation}
  \frakE = \gamma \int_\Omega \left( \frac\delta2 |\GRAD\phi|^2 + \frac1\delta \calW(\phi) \right)
         + \gamma \int_\Gamma \Theta_{fs}(\phi) + \frac12 \int_\Omega \frac1{\vare(\phi)} |\bD|^2
         + \frac12 \int_\Omega \rho(\phi) |\ue|^2 + \frac\lambda2\int_\Omega q^2,
\label{eq:energy}
\end{equation}
where $\vare$ is the electric permittivity of the medium and $\lambda>0$ is a regularization parameter.
Computing the variation of the energy $\frakE$ with respect to $\phi$, while keeping all the other
arguments fixed, we obtain that
\[
  \langle D_\phi \frakE, \bar\phi \rangle = \int_\Omega \mu \bar\phi + \int_\Gamma L \bar\phi,
\]
where $\mu$ is the so-called chemical potential which, in this situation, is given by
\begin{equation}
  \mu = \gamma \left( \frac1\delta \calW'(\phi) - \delta \LAP \phi \right)
      - \frac{ \vare'(\phi) }{2 \vare(\phi)^2 } |\bD|^2 + \frac12 \rho'(\phi) |\ue|^2.
\label{eq:mu}
\end{equation}
The quantity $L$ is given by
\begin{equation}
  L = \gamma \left( \Theta_{fs}'(\phi) + \delta \partial_\bn \phi \right),
\label{eq:L}
\end{equation}
and can be regarded as a ``chemical potential'' on the boundary.
\begin{rem}[Chemical potential]
From the definition of the chemical potential $\mu$ we see that the product $\mu\GRAD\phi$ includes
the usual terms that define the surface tension, \ie
\[
  \gamma \left( \frac1\delta \calW'(\phi) - \delta \LAP \phi \right)\GRAD \phi.
\]
Additionally, it has the term
\[
  -\frac{ \vare'(\phi) }{ 2 \vare(\phi)^2 } |\bD|^2 \GRAD \phi,
\]
which, in some sense, can be thought of as coming from the Maxwell stress tensor.
\end{rem}

With this notation, let us take the time derivative of the free energy:
\[
  \frac{\diff \frakE}{\diff t} = \int_\Omega \mu \phi_t + \int_\Gamma L \phi_t 
    + \int_\Omega \bE \SCAL \bD_t
    + \int_\Omega \rho(\phi) \ue\SCAL\ue_t + \lambda \int_\Omega q q_t,
\]
where $\bE$ is the electric field, defined as $\bE := \vare^{-1}\bD$.
Let us rewrite each one of the terms in this expression.
Using \eqref{eq:phi} and the impermeability conditions,
\[
  \int_\Omega \mu \phi_t = - \int_\Omega \mu \DIV\left( \phi \ue + \bJ_\phi \right) =
      \int_\Omega \GRAD \mu \SCAL \left( \phi \ue + \bJ_\phi \right).
\]
Using \eqref{eq:DMaxwell}
\[
  \int_\Omega \bE \SCAL \bD_t = - \int_\Omega \bE \SCAL \left(  q \ue + \bJ_\bD \right).
\]
For the boundary term, we introduce the material derivative at the boundary
$\dot\phi = \phi_t + \ue_\btau \partial_\btau \phi$ and rewrite
\[
  \int_\Gamma L \phi_t = \int_\Gamma L (\dot\phi - \ue_\btau \partial_\btau \phi).
\]
Notice that $ \sigma(\sigma\ue)_t = \rho \ue_t + \tfrac12 \rho_t \ue$, so that
using \eqref{eq:NSE}, and integrating by parts, we obtain
\[
  \int_\Omega \rho(\phi) \ue \SCAL \ue_t = \int_\Omega \bF \SCAL \ue - \frac12 \int_\Omega \rho(\phi)_t |\ue|^2
        + \int_\Gamma \eta \left(\bS(\ue) \SCAL \bn \right) \SCAL \ue_\btau - \int_\Omega \eta |\bS(\ue)|^2.
\]
Finally, using \eqref{eq:qD} and the impermeability condition
$(q\ue + \bJ_\bD)\SCAL \bn|_\Gamma = 0$,
\[
  \lambda \int_\Omega q q_t = -\lambda \int_\Omega q \DIV\left( q \ue + \bJ_\bD \right) =
      \lambda \int_\Omega \GRAD q \SCAL \left( q \ue + \bJ_\bD\right).
\]

With the help of these calculations, we find that the time-derivative of the free energy can be rewritten as
\begin{equation}
\label{eq:dotE}
\begin{aligned}
  \dot \frakE &= - \int_\Omega \mu \GRAD \phi \SCAL \ue
              + \int_\Omega \bJ_\phi \SCAL \GRAD \mu
              + \int_\Gamma L (\dot\phi - \ue_\btau \partial_\btau \phi)
              - \int_\Omega \bE \SCAL \left( q\ue + \bJ_\bD \right)
              + \int_\Omega \bF \SCAL \ue \\
              &- \frac12 \int_\Omega \rho'(\phi)\phi_t |\ue|^2
              + \int_\Gamma \eta \left(\bS(\ue) \SCAL \bn \right) \SCAL \ue_\btau
              - \int_\Omega \eta |\bS(\ue)|^2
              + \frac\lambda2 \int_\Omega \ue \SCAL \GRAD\left( q^2 \right)
              + \lambda \int_\Omega \GRAD q \SCAL \bJ_\bD.
\end{aligned}
\end{equation}

From \eqref{eq:dotE}, we can identify the power of the system, \ie the time derivative of 
the work $\frakW$ of
internal forces, upon collecting all terms having a scalar product with the velocity $\ue$,
\[
  \dot \frakW = \int_\Omega \bF \SCAL \ue - \int_\Omega \mu \GRAD \phi \SCAL \ue - \int_\Omega q \bE \SCAL \ue
              + \frac\lambda2 \GRAD(q^2) \SCAL \ue - \frac12 \int_\Omega \rho'(\phi)\phi_t \ue\SCAL\ue.
\]
We assume that the system is closed, \ie there are no external forces. This implies that
$\dot\frakW \equiv 0$ and we obtain an expression for the forces $\bF$ acting on the fluid,
\[
  \bF = \mu \GRAD \phi + q \bE + \frac12 \rho'(\phi)\phi_t \ue - \GRAD\left( \frac\lambda2 q^2 \right).
\]

Using the first law of thermodynamics
\[
  \frac{\diff \frakE}{\diff t} = \frac{\diff \frakW}{\diff t} - \calT \frac{\diff \frakS}{\diff t},
\]
where the absolute temperature is denoted by $\calT$ and the entropy by $\frakS$, we can conclude that
\[
  \calT \dot \frakS = \int_\Omega \eta |\bS(\ue)|^2 - \int_\Omega \bE \SCAL \bJ_\bD
                    + \int_\Omega \bJ_\phi \SCAL \GRAD \mu + \lambda \int_\Omega \GRAD q \SCAL \bJ_\bD
                    + \int_\Gamma \eta \left(\bS(\ue) \SCAL \bn \right) \SCAL \ue_\btau
                    + \int_\Gamma L (\dot\phi - \ue_\btau \partial_\btau \phi).
\]

To find an expression for the fluxes we introduce, in the spirit of Onsager \cite{onsager33,MR2261865},
a dissipation function $\Phi$. Since this must be a positive definite function on the fluxes,
the simplest possible expression for a dissipation function is quadratic and diagonal in the fluxes, \eg
\[
  \Phi = \frac12 \int_\Omega \frac1M |\bJ_\phi|^2 + \frac\alpha2 \int_\Gamma {\dot\phi}^2
      + \frac12 \int_\Omega \frac1K |\bJ_\bD|^2 + \frac12 \int_\Gamma \beta |\ue_\btau|^2,
\]
where all the proportionality constants, in principle, can depend on the phase $\phi$. Here, $M$ is
known as the mobility, $K$ the conductivity and $\beta$ the slip coefficient.
Using Onsager's relation
\[
  \left \langle D_\bJ \left(  \dot\frakE(\bJ) + \Phi(\bJ) \right) , \bar\bJ\right\rangle = 0,
  \quad \forall \bar\bJ,
\]
and \eqref{eq:dotE}, we find that
\begin{equation}
\label{eq:fluxes}
  \bJ_\phi = -M \GRAD \mu, \
  \bJ_\phi = -M \GRAD \mu, \
  \bJ_\bD = K \left( \bE - \lambda \GRAD q \right), \
  \beta \ue_\btau = -\eta \bS(\ue)_{\bn\btau} + L \partial_\btau \phi, \
  \alpha\dot\phi = - L,
\end{equation}
where $\bS(\ue)_{\bn\btau} := (\bS(\ue)\SCAL\bn)_\btau$.

\begin{rem}[Constitutive relations]
Definitions \eqref{eq:fluxes} can also be obtained by simply saying that the constitutive 
relations of the fluxes depend 
linearly on the gradients, which is implicitly postulated in the form of the dissipation
function $\Phi$.
\end{rem}

Since, in practical settings, there is an externally applied voltage (which is going to act as the
control mechanism) we introduce a potential $V$ and then the electric field is given by
$\bE = -\GRAD V$ with $V = V_0$ on $\partial^\star\Omega^\star$, where $V_0$ is the voltage applied.

To summarize, we obtain the following system of equations for the phase variable $\phi$ and the
chemical potential $\mu$,
\begin{equation}
  \begin{dcases}
    \phi_t + \ue \SCAL \GRAD \phi = \DIV( M(\phi) \GRAD \mu ), & \text{in }\Omega, \\
    \mu = \gamma \left( \frac1\delta \calW'(\phi) - \delta \LAP \phi \right)
    - \frac12\vare'(\phi)|\GRAD V|^2 + \frac12 \rho'(\phi)|\ue|^2, & \text{in } \Omega, \\
    \alpha\left( \phi_t + \ue_\btau \partial_\btau \phi \right)
          + \gamma\left( \Theta_{fs}'(\phi) + \delta \partial_\bn \phi \right) = 0,\
    M(\phi) \partial_n \mu = 0, & \text{on }\Gamma,
  \end{dcases}
\label{eq:phase}
\end{equation}
and the velocity $\ue$ and pressure $\pe$,
\begin{equation}
  \begin{dcases}
    \frac{ D(\rho(\phi)\ue) }{Dt} - \DIV\left( \eta(\phi) \bS(\ue) \right) + \GRAD \pe = \mu \GRAD \phi
      - q \GRAD \left( V + \lambda q \right)
        + \frac12 \rho'(\phi)\phi_t \ue & \text{on }\Omega, \\
    \DIV \ue =0, & \text{in } \Omega, \\
    \ue\SCAL\bn = 0, & \text{on }\Gamma, \\
    \beta(\phi) \ue_\btau + \eta(\phi) \bS(\ue)_{\bn\btau}
        = \gamma\left( \Theta_{fs}'(\phi) + \delta \partial_\bn \phi \right) \partial_\btau \phi,
        & \text{on }\Gamma,
  \end{dcases}
\label{eq:vel}
\end{equation}
where we have set
\[
  \frac{D(\rho(\phi)\ue)}{Dt} :=
  \sigma(\phi)(\sigma(\phi) \ue)_t + \rho(\phi) \ue\ADV \ue + \tfrac12 \DIV(\rho(\phi) \ue)\ue.
\]
In addition, we have the equation for the electric charges $q$,
\begin{equation}
  \begin{dcases}
    q_t + \DIV (q \ue ) = \DIV \left[K(\phi) \GRAD\left( \lambda  q + V \right)\right], & \text{in }\Omega, \\
    K(\phi) \GRAD\left( \lambda q + V \right) \SCAL \bn = 0, & \text{on }\Gamma,
  \end{dcases}
\label{eq:charge}
\end{equation}
and voltage $V$,
\begin{equation}
  \begin{dcases}
    - \DIV\left( \vare^\star(\phi) \GRAD V \right) = q \chi_\Omega,
      & \text{in }\Omega^\star, \\
    V = V_0, & \text{on }\partial^\star\Omega^\star, \\
    \partial_\bn V = 0, & \text{on }\partial\Omega^\star \cap \Gamma,
  \end{dcases}
\label{eq:potential}
\end{equation}
where
\[
  \vare^\star(\phi) =
    \begin{dcases}
      \vare(\phi), & \Omega, \\
      \vare_D, & \Omega^\star \setminus \Omega,
    \end{dcases}
\]
with $\vare_D$ being the value of the permittivity on the dielectric plates 
$\Omega^\star \setminus \Omega$,
so $\vare_D$ is constant there.

\begin{rem}[Generalized Navier Boundary condition]
In \eqref{eq:vel}, the boundary condition for the tangential velocity is known as the generalized Navier
boundary condition (GNBC), and it is aimed at resolving the so-called contact line paradox of the movement
of a two phase fluid on a solid wall.
The reader is referred to, for instance, \cite{QianCiCP,MR2261865,MR2498521} for a discussion of
its derivation.
Although there has been a lot of discussion and controversy around the validity of this boundary condition,
see for instance \cite{Buscaglia20113011,MR2455379}, we shall take the GNBC as a given and will not
discuss its applicability and/or consequences here.
\end{rem}

\subsection{Nondimensionalization}
\label{sub:non-dim}

\begin{table}[h]
  \caption{Physical Parameters at standard temperature (25$^\circ$ C) and pressure (1 bar),
  taken from \cite{CRC_Book:2002}.
  A Farad (F) is $\mathrm{C}^2 / \mathrm{J}$.  For drinking water, $K$ is $5\cdot10^{-4}$ to $5\cdot10^{-2}$.}
  \begin{centering}
  \begin{tabular}{|c||c|}
    \hline
    \textbf{Parameter} & \textbf{Value} \\
    \hline\hline
    Surface Tension $\gamma$ & (air/water) 0.07199 $\mathrm{J} / \mathrm{m}^{2}$ \\
    \hline
    Dynamic Viscosity $\etascale$ & (water) $8.68\cdot10^{-4}$,
    (air) $1.84\cdot10^{-5}$ $\;\mathrm{Kg} / \mathrm{m\cdot s}$ \\ \hline

    Density $\rhoscale$ & (water) 996.93, (air) 1.1839 $\;\mathrm{Kg} / \mathrm{m}^{3}$ \\ \hline

    Length Scale (Channel Height) $\Lscale$ & $50\cdot10^{-6}$ to $100\cdot10^{-6}$ m \\ \hline

    Velocity Scale $\uscale$ & 0.001 to 0.05 $\mathrm{m} / \mathrm{s}$ \\ \hline

    Voltage Scale $\Vscale$ & 10 to 50 Volts \\ \hline

    Permittivity of Vacuum $\epsvac$ & $8.854\cdot10^{-12}$ $~\mathrm{F} / \mathrm{m}$ \\ \hline

    Permittivity $\epsscale$ & (water) $78.36 \cdot \epsvac$, (air) $1.0 \cdot \epsvac$ \\ \hline

    Charge (Regularization) Parameter $\lambda$ & 0.5 $~\mathrm{J} \cdot \mathrm{m}^{3} / \mathrm{C}^2$ \\ \hline

    Mobility $\Mscale$ & 0.01 $\; \mathrm{m}^5 / (\mathrm{J} \cdot \mathrm{s}) $ \\ \hline

    Phase Field Parameter $\alpha$ & 0.001 $\; \mathrm{J} \cdot \mathrm{s} / \mathrm{m}^2 $ \\ \hline

    Electrical Conductivity $\Kscale$ & (deionized water) $5.5\cdot10^{-6}$,  \\
                                & (air) $\approx$ 0.0
                                $\; \mathrm{C}^2 / (\mathrm{J} \cdot \mathrm{m} \cdot \mathrm{s}) \equiv$
                                Amp$/ (\mathrm{Volt} \cdot \mathrm{m})$ \\
    \hline\hline
  \end{tabular}
  \end{centering}
\label{tbl:Physical_Parameters}
\end{table}

Here we present appropriate scalings so that we may write equations \eqref{eq:phase}--\eqref{eq:potential}
in non-dimensional form. Table~\ref{tbl:Physical_Parameters} shows some typical values
for the material parameters appearing in the model. Consider the following scalings:
\begin{align*}
  \tilde{\rho} &= \rho / \rhoscale \text{ (choose $\rhoscale$)},
  & \tilde{\eta} &= \eta / \etascale \text{ (choose $\etascale$)},
  &\tilde{\beta} &= \beta / \betascale,
  & \betascale &= \etascale / \Lscale, \\
  \tilde{\pe}  &= \pe / \pscale,
  &\pscale &= \rhoscale \uscale^2,
  &\tilde{\ue}  &= \ue / \uscale \text{ (choose $\uscale$)},
  &\tilde{\bx}  &= \bx / \Lscale \text{ (choose $\Lscale$)}, \\
  \tilde{t} &= t / \tscale,
  &\tscale &= \Lscale / \uscale,
  &\tilde{\mu} &= \mu / \muscale,
  &\muscale &= \gamma / \Lscale, \\
  \tilde{q} &= q / \qscale,
  &\qscale &= \Vscale / \lambda,
  &\tilde{V} &= V / \Vscale, \text{ (choose $\Vscale$)},
  & \tilde{\vare} &= \vare / \epsscale, \\
  \tilde{\delta} &= \delta / \Lscale,
  &\widetilde{M} &= M / \Mscale,
  &\widetilde{K} &= K / \Kscale,
  &\CAP &= \frac{\etascale \uscale}{\gamma}, \\
  \REY &= \frac{\rhoscale \uscale \Lscale}{\etascale},
  &\WEB &= \frac{\rhoscale \uscale^2 \Lscale}{\gamma},
  &\BOEW &= \frac{\epsscale \Vscale^2}{\Lscale \gamma},
  &\IE &= \frac{\rhoscale \uscale^2}{\qscale \Vscale}, \\
  \STPH &= \frac{\gamma}{\alpha / \tscale},
  &\MO &= \frac{\gamma \Mscale}{\Lscale^2 \uscale},
  &\KO &= \frac{\Vscale \Kscale}{\Lscale \qscale \uscale},
  &\CH &= \frac{\qscale \Lscale^2}{\Vscale \epsscale}, \\
 \end{align*}
where $\CAP$ is the capillary number, $\REY$ is the Reynolds number, $\WEB$ is the Weber 
number, $\BOEW$ is the electro-wetting Bond number, $\IE$ is the ratio of fluid forces 
to electrical forces, $\STPH$ is the ratio of surface tension to ``phase field forces,'' 
$\MO$ is a (non-dimensional) mobility coefficient, $\KO$ is a conductivity coefficient, and 
$\CH$ is an electric charge coefficient.

Let us now make the change of variables. To simplify notation, we drop the tildes,
and consider all variables and differential operators as non-dimensional. The fluid equations read:
\begin{equation*}
  \begin{dcases}
    \frac{ D(\rho \ue) }{Dt}
    -\frac{1}{\REY} \DIV\left( \eta \bS(\ue) \right)
    +\GRAD \pe =
    \frac{1}{\WEB} \mu \GRAD \phi
    - \frac{1}{\IE} q \GRAD \left( V + q \right)
    + \frac12 \rho'(\phi)\phi_t \ue,
      & \text{in }\Omega, \\
    \DIV \ue =0, & \text{in } \Omega, \\
    \ue\SCAL\bn = 0, & \text{on }\Gamma, \\
    \beta \ue_\btau
    + \eta \bS(\ue)_{\bn\btau} =
    \frac1{\CAP} \left( \Theta_{fs}'(\phi) + \delta \partial_\bn \phi \right) \partial_\btau \phi,
        & \text{on }\Gamma.
  \end{dcases}
\end{equation*}
The phase-field equations change to (again dropping the tilde)
\begin{equation*}
  \begin{dcases}
    \phi_t + \ue \SCAL \GRAD \phi = \MO \DIV( M(\phi)\GRAD \mu),
      & \text{in }\Omega, \\
    \mu = \left( \frac1\delta \calW'(\phi) - \delta \LAP \phi \right)
    -\BOEW \frac12 \vare'(\phi)|\GRAD V|^2
    + \WEB \frac12 \rho'(\phi)|\ue|^2, & \text{in } \Omega, \\
    \phi_t + \ue_\btau \partial_\btau \phi
    + \STPH \left(
      \Theta_{fs}'(\phi) + \delta \partial_\bn \phi \right) = 0,
    \
    \partial_n \mu = 0, & \text{on }\Gamma.
  \end{dcases}
\end{equation*}
Performing the change of variables on the charge transport equation gives
\begin{equation*}
  \begin{dcases}
    q_t + \DIV (q \ue )
    = \KO \DIV \left(K(\phi)\GRAD\left(q + V \right)\right),
      & \text{in }\Omega, \\
    \bn \SCAL \GRAD\left(q + V \right) = 0, & \text{on }\Gamma.
  \end{dcases}
\end{equation*}
Lastly, for the electrostatic equation we obtain
\begin{equation*}
  \begin{dcases}
    - \DIV\left( (\phi) \GRAD V \right) =
    \CH q \chi_\Omega,
      & \text{in }\Omega^\star, \\
    V = V_0 / \Vscale, & \text{on }\partial^\star\Omega^\star, \\
    \partial_\bn V = 0, & \text{on }\partial\Omega^\star \cap \Gamma.
  \end{dcases}
\end{equation*}
where $\vare^\star(\phi)$ has been normalized by $\epsscale$.

To alleviate the notation, for the rest of our discussion we will set all the nondimensional groups
($\CAP$, $\REY$, $\WEB$, $\BOEW$, $\IE$, $\STPH$, $\MO$, $\KO$ and $\CH$)
to one. If needed, the dependence of the constants on all these parameters can be traced by following our
arguments.
Moreover, we must note that
if a simplification of this model is desired, then these scalings must serve as a guide to
decide which effects are dominant.

\subsection{Tangential Derivatives at the Boundary}
\label{sub:tangentialderiv}
As we can see from \eqref{eq:phase} and \eqref{eq:vel}, our model incorporates tangential derivatives of the
phase variable $\phi$ at the boundary $\Gamma$. Unfortunately, in the analysis, we are not capable of dealing
with these terms. Therefore, we propose some simplifications.

The first possible simplification is simply to ignore the terms that contain this tangential derivative; see
\cite{MR2511642}. However, it is our feeling that the presence of them is important, specially in dealing
with the contact angle in the GNBC.

A second possibility would be to add an \emph{ad hoc} term of the form $\LAP_\Gamma \phi$ on the
boundary condition
for the phase variable, where by $\LAP_\Gamma$ we denote the Laplace-Beltrami operator on $\Gamma$.
A similar approach has
been followed, in a somewhat different context, for instance, by Pr\"uss \etal \cite{MR2230586}
and Cherfils \etal \cite{MR2629535}. However, this condition might lead to lack of conservation of $\phi$,
which is an important feature of phase field models based on the Cahn Hilliard equation.

Finally, the approach that we propose is to recall that, in principle, the phase field variable must be
constant in the bulk of each one of the phases and so $\partial_\btau \phi \approx 0$ there. Moreover,
in the sharp interface limit this tangential derivative must be a Dirac measure supported on the interface.
Therefore we define a function
\begin{equation}
  \psi(\phi) = \frac{1}{\Lscale} \frac1\delta e^{-\frac{\phi^2}{2\delta}},
  \quad \text{where } \delta \text{ is non-dimensional},
\label{eq:defpsi}
\end{equation}
and replace all the instances of $\partial_\btau \phi$ by $\psi(\phi)$.

\subsection{Contact Line Pinning}
\label{sub:pinning}
Simply put, the contact line pinning (hysteresis) is a frictional effect that occurs at the three-phase
contact line, and is rather controversial. We refer the reader to \cite{walker:102103,MR2595379} 
for an explanation about its origins
and possible dependences. Let us here only mention that, macroscopically, the pinning force has a
threshold value and, thus, it should depend on the stress at the contact line.
It is important to take into account contact line pinning since, as observed in
\cite{walker:102103,MR2595379}, it is crucial  for capturing the true time scales of the problem.

We propose a phenomenological approach to deal with this effect. From the GNBC,
\[
  \beta \ue_\btau + \eta \bS(\ue)_{\bn\btau}
  = \gamma\left( \Theta_{fs}'(\phi) + \delta \partial_\bn \phi \right)\psi(\phi),
\]
we can see that, to recover no-slip conditions, one must set the slip coefficient $\beta$ sufficiently large.
On the contrary, when $\beta$ is small, one obtains an approximation of full slip conditions.
A simple dimensional argument then shows that $\beta = \eta\ell$, where $\ell$ has
the dimensions of inverse length.
Therefore, we propose the slip coefficient to have the following form
\[
  \beta = \eta(\phi) \ell(\phi,\bS),
\]
where
\[
  \ell(\phi,\bS) = \frac{1}{\Lscale}
  \begin{dcases}
    \frac1\delta, & |\phi|>\frac12, \\
    \frac1\delta, & |\phi| \leq \frac12, \text{ and } |\bS(\ue)_{\bn\btau}| \ll T_p, \\
    1, & |\phi| \leq \frac12, \text{ and } |\bS(\ue)_{\bn\btau}| \approx T_p, \\
    \delta, & |\phi| \leq \frac12, \text{ and } |\bS(\ue)_{\bn\btau}| \gg T_p,
  \end{dcases}
\]
where $\delta$ is the non-dimensional transition length.  For the purposes of analysis, 
we face the same difficulties in this expression as in \S\ref{sub:tangentialderiv}.  Ergo,
we will use this to model pinning in the numerical examples, but leave it out of the analysis.

\section{Formal Weak Formulation and Formal Energy Estimate}
\label{sec:Energy}
In this section we obtain a weak formulation for problem \eqref{eq:phase}--\eqref{eq:potential} and
show a formal energy estimate, which serves as an a priori estimate and the basic relation on which
our existence theory is based.

\subsection{Formal Weak Formulation}
\label{sub:weak}
To obtain a weak formulation of the problem, we begin by multiplying the first equation of
\eqref{eq:phase} by $\bar\phi$, the second by $\bar\mu$ and integrating in $\Omega$.
After integration by parts, taking into account the boundary conditions, we arrive at
\begin{subequations}
\label{eq:phaseweak}
  \begin{equation}
  \label{eq:phaseweak1}
    \scl \phi_t , \bar\phi \scr
    + \scl \ue\SCAL\GRAD\phi, \bar\phi \scr
    + \scl M(\phi) \GRAD\mu, \GRAD\bar\phi \scr =0,
  \end{equation}
\text{and}
  \begin{multline}
    \scl \mu, \bar\mu \scr =
    \frac\gamma\delta \scl \calW'(\phi), \bar\mu \scr
    + \gamma \delta \scl \GRAD\phi, \GRAD\bar\mu \scr
    - \frac12 \scl \vare'(\phi) |\GRAD V|^2, \bar\mu \scr
    + \frac12 \scl \rho'(\phi) |\ue|^2,  \bar \mu \scr \\
    + \alpha \sbl \phi_t + \ue_\btau \psi(\phi), \bar\mu \sbr
    + \gamma \sbl \Theta_{fs}'(\phi), \bar\mu \sbr.
  \label{eq:phaseweak2}
  \end{multline}
\end{subequations}
Multiply the first equation of \eqref{eq:vel} by $\bw$ such that $\bw\SCAL\bn|_\Gamma = 0$, the second
by $\bar p$ and integrate in $\Omega$. Integration by parts on the first equation,
in conjunction with the boundary conditions and \eqref{eq:defpsi}, yields
\begin{align*}
  -\scl \DIV (\eta(\phi) \bS(\ue)), \bw \scr &=
  \scl \eta(\phi) \bS(\ue),\bS(\bw) \scr
  - \sbl \eta(\phi) \bS(\ue)_\bn, \bw_\btau  \sbr \\
  & = \scl \eta(\phi) \bS(\ue), \bS(\bw) \scr
  + \sbl \beta(\phi) \ue_\btau,  \bw_\btau \sbr
  - \gamma \sbl \Theta_{fs}'(\phi) + \delta\partial_\bn \phi, \bw_\btau \psi(\phi) \sbr \\
  &= \scl \eta(\phi) \bS(\ue),\bS(\bw) \scr
  + \sbl \beta(\phi) \ue_\btau,  \bw_\btau \sbr
  + \alpha \sbl \phi_t + \ue_\btau \psi(\phi), \bw_\btau \psi(\phi) \sbr,
\end{align*}
where we used the third equation of \eqref{eq:phase}. With these manipulations we obtain
\begin{subequations}
\label{eq:NSEweak}
  \begin{multline}
  \label{eq:velweak}
    \scl \frac{ D( \rho(\phi) \ue ) }{Dt}, \bw \scr + \scl \eta(\phi) \bS(\ue), \bS(\bw) \scr
    - \scl \pe, \DIV \bw \scr
    + \sbl \beta(\phi) \ue_\btau, \bw_\btau \sbr
    + \alpha \sbl \ue_\btau \psi(\phi), \bw_\btau \psi(\phi) \sbr \\
    = \scl \mu \GRAD\phi, \bw \scr
    - \scl q \GRAD(\lambda q + V), \bw \scr
    + \frac12 \scl \rho'(\phi)\phi_t \ue, \bw \scr
    - \alpha \sbl \phi_t \psi(\phi), \bw_\btau \sbr,
  \end{multline}
\text{for all $\bw$, and }
  \begin{equation}
      \scl \bar p, \DIV \ue \scr = 0,
  \label{eq:incweak}
  \end{equation}
\end{subequations}
for all $\bar p$. Multiply \eqref{eq:charge} by $r$ and integrate in $\Omega$ to get
\begin{equation}
\label{eq:chargeweak}
  \scl q_t,  r \scr - \scl q\ue, \GRAD r \scr + \scl K(\phi) \GRAD(\lambda q + V ), \GRAD r \scr = 0.
\end{equation}
Let $W$ be a function that equals zero on $\partial^\star \Omega^\star$.
Multiply the equation for the electric potential \eqref{eq:potential} by $W$, integrate in $\Omega^\star$
to obtain
\begin{equation}
\label{eq:potentialweak}
  \scl \vare^\star(\phi) \GRAD V,  \GRAD W \scr_{\Omega^\star}= \scl q, W \scr.
\end{equation}

Given the way the model has been derived, it is clear that an energy estimate must exist. Before we obtain it
let us show a comparison result \emph{\`a la} Gr\"onwall.

\begin{lem}[Gr\"onwall]
\label{lem:mod-gronwall}
Let $f,g,h,w:[0,T] \rightarrow \Real$ be measurable and positive functions such that
\begin{equation}
\label{eq:gR}
  f(t)^2 + \int_0^t g(s)\diff s \leq h(t) + \int_0^t f(s)w(s)\diff s, \quad \forall t\in[0,T].
\end{equation}
Then
\[
  \sup_{s\in[0,T]}f(s)^2 + \frac12 \int_0^T g(s) \diff s \leq 4 \sup_{s\in[0,T]}h(s) 
  + 4T \int_0^T w^2(s) \diff s,
  \quad \forall t\in[0,T]
\]
\end{lem}
\begin{proof}
Take, in \eqref{eq:gR}, $t=t_0$, where
\[
  t_0 = \argmax \left\{ f(s) : s \in [0,T] \right\},
\]
then
\[
  f(t_0)^2 + \int_0^{t_0} g(s) \diff s \leq \max_{s\in[0,T]}h(s) + f(t_0)\int_0^{t_0} w(s) \diff s
  \leq \max_{s\in[0,T]}h(s) + \frac12 f(t_0)^2 + \left( \int_0^T w(s) \diff s \right)^2.
\]
Canceling the common factors, applying the Cauchy Schwarz inequality on the right and taking the supremum
on the left hand side we obtain the result.
\end{proof}
\begin{rem}[Exponential in time estimates]
The main advantage of using Lemma~\ref{lem:mod-gronwall} to obtain \emph{a priori} estimates, as opposed
to a standard argument invoking Gr\"onwall's inequality, is that we can avoid exponential dependence on
the final time $T$.
\end{rem}

The following result provides the formal energy estimate.
\begin{thm}[Stability]
\label{thm:energy}
If there is a solution to \eqref{eq:phase}--\eqref{eq:potential}, then it must satisfy the following estimate
\begin{multline}
  \sup_{s\in(0,T]}\left\{
    \int_\Omega \left[\frac12 \rho(\phi) |\ue|^2 + \frac\lambda4 q^2
    + \gamma\left( \frac\delta2 |\GRAD \phi|^2 + \frac1\delta \calW(\phi)\right) \right]
    + \int_{\Omega^\star} \frac14 \vare^\star(\phi) |\GRAD V|^2
    + \gamma \int_\Gamma \Theta_{fs}(\phi)
  \right\} \\
  + \int_0^T\left\{
    \int_\Omega \left[
      \eta(\phi) |\bS(\ue)|^2 + M(\phi) |\GRAD \mu|^2 + K(\phi) |\GRAD(\lambda q + V )|^2
    \right]
    + \int_\Gamma \left[
      \beta(\phi) |\ue_\btau|^2 + \alpha |\phi_t + \ue_\btau \psi(\phi)|^2
    \right]
  \right\} \leq \\
  \left\{
    \int_\Omega \left[\frac12 \rho(\phi) |\ue|^2 + q^2
      + \gamma\left( \frac\delta2 |\GRAD \phi|^2 + \frac1\delta \calW(\phi)\right)
      + \frac12 |\bar V_0|^2
    \right]
    + \int_{\Omega^\star}\left(
      \vare^\star(\phi) |\GRAD V|^2 + \vare_M |\GRAD \bar V_0|^2
    \right)
  \right. \\
  \left.
    + \gamma\int_\Gamma \Theta_{fs}(\phi)
  \right\}\Big|_{t=0}
  +\sup_{s\in[0,T]}\left\{
    \int_{\Omega^\star}\vare_M |\GRAD \bar V_0|^2
    +\int_\Omega \frac1\lambda |\bar V_0|^2(t)
  \right\} \\
  +cT\int_0^T \left[
    \int_{\Omega^\star} \vare_M |\GRAD \bar V_{0,t}|^2
    + \frac4\lambda \int_\Omega |\bar V_{0,t}|^2
  \right],
\label{eq:Elaw}
\end{multline}
where $c$ does not depend on $T$.
\end{thm}
\begin{proof}
We first deal with the Navier Stokes and Cahn Hilliard equations in a way very similar to
Theorem~3.1 of \cite{SalgadoMCL}.
Set $\bw = \ue$ in \eqref{eq:velweak} and notice that
\[
  \scl \frac{ D(\rho\ue)}{Dt}, \ue \scr = \frac12\frac{\diff}{\diff t}\int_\Omega \rho |\ue|^2,
\]
because
\[
  \frac{ D(\rho \ue)}{Dt} = \sigma(\sigma\ue)_t + \rho \ue \ADV \ue + \frac12 \DIV(\rho \ue) \ue.
\]
We obtain
\begin{multline}
  \frac{\diff}{\diff t} \frac12\int_\Omega \rho |\ue|^2 + \int_\Omega \eta |\bS(\ue)|^2
        + \int_\Gamma \beta(\phi) |\ue_\btau|^2 + \alpha \int_\Gamma |\ue_\btau \psi(\phi)|^2
    = \scl \mu \GRAD \phi, \ue \scr \\
    - \scl q \GRAD (\lambda q + V ), \ue \scr
    + \frac12 \scl \rho'(\phi)\phi_t, |\ue|^2 \scr
    - \alpha \sbl \phi_t \ue_\btau, \psi( \phi) \sbr.
\label{eq:velE}
\end{multline}
Set $\bar\phi = \mu$ in \eqref{eq:phaseweak1} to get
\begin{equation}
  \scl \mu, \phi_t \scr + \scl \mu \GRAD \phi, \ue \scr + \int_\Omega M(\phi) |\GRAD \mu|^2 = 0.
\label{eq:phaseE}
\end{equation}
Set $\bar\mu = -\phi_t$ in \eqref{eq:phaseweak2} to write
\begin{multline}
  - \scl \phi_t, \mu \scr
    = - \gamma \frac{\diff}{\diff t}\left[
      \int_\Omega \left (\frac\delta2 |\GRAD \phi|^2 + \frac1\delta \calW(\phi) \right)
      + \int_\Gamma \Theta_{fs}(\phi)
    \right]
    + \frac12 \scl  \vare'(\phi)\phi_t, |\GRAD V|^2 \scr \\
    - \frac12 \scl  \rho'(\phi)\phi_t, |\ue|^2 \scr
    - \alpha \int_\Gamma (\phi_t)^2
    - \alpha \sbl \phi_t, \ue_\btau \psi( \phi) \sbr.
\label{eq:chemE}
\end{multline}
Add \eqref{eq:velE}, \eqref{eq:phaseE} and \eqref{eq:chemE} to arrive at
\begin{multline}
  \frac{\diff}{\diff t}\left[ \int_\Omega  \left(\frac12 \rho(\phi) |\ue|^2 +
          \gamma \left( \frac\delta2 |\GRAD \phi|^2 + \frac1\delta \calW(\phi) \right) \right)
          + \gamma \int_\Gamma \Theta_{fs}(\phi)
                                    \right]
  + \int_\Omega \eta(\phi) |\bS(\ue)|^2 \\
  + \int_\Gamma \beta(\phi) |\ue_\btau|^2
  + \int_\Omega M(\phi) |\GRAD \mu|^2
  +\alpha \int_\Gamma \left( \phi_t + \ue_\btau \psi\left( \phi \right) \right)^2 =
  - \scl q \GRAD (\lambda q + V ), \ue \scr
  + \frac12 \scl \vare'(\phi)\phi_t, |\GRAD V|^2 \scr.
\label{eq:combined}
\end{multline}

We next deal with the electrostatic equations. Set $r = \lambda q + V$ in \eqref{eq:chargeweak}
to get
\begin{equation}
\label{eq:chargeE}
  \frac\lambda2 \frac{\diff}{\diff t}\int_\Omega q^2 + \scl V, q_t \scr
  - \scl q \GRAD ( \lambda q + V ),  \ue \scr + \int_\Omega K(\phi) | \GRAD (\lambda q + V ) |^2 = 0.
\end{equation}
Take the time derivative of \eqref{eq:potentialweak} and set $W=V-\bar V_0$, where by $\bar V_0$ we mean
an extension of $V_0$ to $\Omega^\star$. We obtain
\begin{equation}
  \int_{\Omega^\star} \partial_t \left( \vare^\star(\phi) \right) |\GRAD V|^2
  + \frac12 \int_{\Omega^\star}\vare^\star(\phi) \partial_t \left( |\GRAD V|^2 \right)
  = \scl q_t, V \scr -\scl q_t, \bar V_0 \scr
  + \scl \partial_t(\vare^\star(\phi) \GRAD V), \GRAD \bar V_0 \scr_{\Omega^\star}.
\label{eq:potentialE}
\end{equation}
Add \eqref{eq:combined}, \eqref{eq:chargeE} and \eqref{eq:potentialE}  and recall 
that $\vare^\star(\phi)$ is constant on $\Omega^\star \setminus \Omega$. We thus obtain
\begin{multline*}
  \frac\diff{\diff t}\left\{
    \int_\Omega \left[\frac12 \rho(\phi) |\ue|^2 + \frac\lambda2 q^2
    + \gamma\left( \frac\delta2 |\GRAD \phi|^2 + \frac1\delta \calW(\phi)\right) \right]
    + \int_{\Omega^\star} \frac12 \vare^\star(\phi) |\GRAD V|^2 + \gamma \int_\Gamma 
    \Theta_{fs}(\phi)
  \right\} \\
  + \int_\Omega \left[\eta(\phi) |\bS(\ue)|^2
  +  M(\phi) |\GRAD \mu|^2 + K(\phi) |\GRAD(\lambda q + V )|^2 \right]
  + \int_\Gamma \left[\beta(\phi) |\ue_\btau|^2 + \alpha |\phi_t + \ue_\btau \psi(\phi)|^2 \right] \\
  = \scl \partial_t (\vare^\star(\phi) \GRAD V), \GRAD \bar V_0 \scr_{\Omega^\star}
  - \scl q_t, \bar V_0 \scr.
\end{multline*}
Integrate in time over $[0,t]$, with $0<t<T$ and integrate by parts the right hand side. Repeated
applications of the Cauchy-Schwarz inequality give us
\begin{multline*}
   \left\{
    \int_\Omega \left[\frac12 \rho(\phi) |\ue|^2 + \frac\lambda4 q^2
    + \gamma\left( \frac\delta2 |\GRAD \phi|^2 + \frac1\delta \calW(\phi)\right) \right]
    + \int_{\Omega^\star} \frac14 \vare^\star(\phi) |\GRAD V|^2 + \gamma \int_\Gamma \Theta_{fs}(\phi)
  \right\}\Big|_{t} \\
  + \int_0^t\int_\Omega \left[ \eta(\phi) |\bS(\ue)|^2
  + M(\phi) |\GRAD \mu|^2 + K(\phi) |\GRAD(\lambda q + V )|^2 \right]
  + \int_0^t\int_\Gamma \left[\beta(\phi) |\ue_\btau|^2 + \alpha |\phi_t + \ue_\btau \psi(\phi)|^2 \right]
  \\
  \leq \left\{
    \int_\Omega \left[\frac12 \rho(\phi) |\ue|^2 + q^2
    + \gamma\left( \frac\delta2 |\GRAD \phi|^2 + \frac1\delta \calW(\phi)\right) \right]
    + \int_{\Omega^\star} \vare^\star(\phi) |\GRAD V|^2 + \gamma \int_\Gamma \Theta_{fs}(\phi)
  \right\}\Big|_{t=0}\\
  + \int_{\Omega^\star}
    \vare_M \left(
      |\GRAD \bar V_0|^2(t) + |\GRAD \bar V_0|^2(0)
    \right)
  +\int_\Omega \left[
    \frac1\lambda |\bar V_0|^2(t) + \frac12 |\bar V_0|^2(0)
  \right] \\
  + c\int_0^t \left\{
    \int_{\Omega^\star} \vare_M |\GRAD \bar V_{0,t}|^2  +
    \frac4\lambda \int_\Omega |\bar V_{0,t}|^2
  \right\}^{1/2}
  \left[
    \frac\lambda4 \int_\Omega q^2 
    + \int_{\Omega^\star} \frac14 \vare^\star(\phi) |\GRAD V|^2
  \right]^{1/2},
\end{multline*}
where $\vare_M$ is the maximal value of the function $\vare^\star(\phi)$.

Finally, if we set
\begin{align*}
  f(t) &= \left\{
      \int_\Omega \left[\frac12 \rho(\phi) |\ue|^2 + \frac\lambda4 q^2
      + \gamma\left( \frac\delta2 |\GRAD \phi|^2 + \frac1\delta \calW(\phi)\right) \right]
      + \int_{\Omega^\star} \frac14 \vare^\star(\phi) |\GRAD V|^2 
      + \gamma \int_\Gamma \Theta_{fs}(\phi)
    \right\}(t), \\
  g(t) &= \int_\Omega \left[
      \eta(\phi) |\bS(\ue)|^2 + M(\phi) |\GRAD \mu|^2 + K(\phi) |\GRAD(\lambda q + V )|^2
    \right]
    + \int_\Gamma \left[
      \beta(\phi) |\ue_\btau|^2 + \alpha |\phi_t + \ue_\btau \psi(\phi)|^2
    \right](t), \\
  h(t) &= \left\{
      \int_\Omega \left[\frac12 \rho(\phi) |\ue|^2 + q^2
      + \gamma\left( \frac\delta2 |\GRAD \phi|^2 + \frac1\delta \calW(\phi)\right) \right]
      + \int_{\Omega^\star} \vare^\star(\phi) |\GRAD V|^2 + \gamma \int_\Gamma \Theta_{fs}(\phi)
    \right\}\Big|_{0}, \\
    &+ \int_{\Omega^\star}\vare_M \left(
      |\GRAD \bar V_0|^2(t) + |\GRAD \bar V_0|^2(0)
    \right)
    +\int_\Omega \left[
      \frac1\lambda |\bar V_0|^2(t) + \frac12 |\bar V_0|^2(0)
    \right], \\
  w(t) &= \left\{
    \int_{\Omega^\star} \vare_M |\GRAD \bar V_{0,t}|^2 +
    \frac4\lambda \int_\Omega |\bar V_{0,t}|^2 
  \right\}^{1/2},
\end{align*}
then an application of Lemma~\ref{lem:mod-gronwall} gives the desired estimate.
\end{proof}

\section{The Fully Discrete Problem and Its Analysis}
\label{sec:Discrete}
In this section we introduce a space-time discrete problem that is used to
approximate the electrowetting problem \eqref{eq:phaseweak}--\eqref{eq:potentialweak}.
Using this discrete problem, and the result of Theorem~\ref{thm:energy}, we will prove that a
time-discrete version of our problem always has a solution. Moreover, in Section~\ref{sec:NumExp}, we
will base our numerical experiments on a variant of the problem defined here.

\subsection{Definition of the Fully Discrete Problem}
\label{sub:defhtau}
To discretize in time, as discussed in \S\ref{sub:Notation}, we divide the time interval $[0,T]$ into
subintervals of length $\dt>0$. Recall that the time increment operator $\frakd$
was introduced in \eqref{eq:frakd} and the time average operator $\overline{(\cdot)}$ in \eqref{eq:defstar}.

To discretize in space, we introduce a parameter
$h>0$ and let $\polW_h \subset \Hunstar$, $\polQ_h \subset \Hun$, $\polX_h \subset \bV$ and
$\polM_h \subset \tildeLdeux$ be finite dimensional subspaces.
We require the following compatibility condition between the spaces $\polW_h$ and
$\polQ_h$:
\begin{equation}
  W_h|_{\Omega} \in \polQ_h, \quad \forall W_h \in \polW_h.
\label{eq:voltchargecomp}
\end{equation}
Moreover, we require that the pair of spaces $(\polX_h,\polM_h)$ satisfies the so-called
LBB condition (see \cite{GR86,BF91,MR2050138}), that is, there exists a constant $c$ independent of
$h$ such that
\begin{equation}
  c \| \bar{p}_h \|_{L^2} \leq \sup_{\bv_h \in \polX_h}
  \frac{ \int_\Omega \bar{p}_h \DIV \bv_h }{ \| \bv_h \|_{\bH^1}},
  \quad \forall \bar{p}_h \in \polM_h.
\label{eq:LBB}
\end{equation}
Finally, we assume that if $\polY$ is any of the continuous spaces and $\polY_h$ the corresponding subspace,
then $h_1 < h_2$ implies $\polY_{h_2} \subset \polY_{h_1}$. Moreover, the family of spaces
$\{ \polY_h\}_{h>0}$, is ``dense in the limit.''
In other words, for every $h>0$ there is a continuous operator $\calI_h : \polY \rightarrow \polY_h$ such
that when $h\rightarrow 0$
\[
  \| y - \calI_h y \|_{\polY} \rightarrow 0, \quad \forall y \in \polY.
\]

The space $\polW_h$ will be used to approximate the voltage; $\polQ_h$ the charge, phase field and
chemical potential; and $\polX_h,\ \polM_h$ the velocity and pressure, respectively. Finally, to
account for the boundary conditions on the voltage, we denote
\[
  \polW_h(\bar V_0^{k+1}) = \polW_h + \bar V_0^{k+1}.
\]

\begin{rem}[Finite elements]
The introduced spaces can be easily constructed using, for instance, finite elements, see
\cite{GR86,BF91,MR2050138,Ci78} for details. The compatibility condition \eqref{eq:voltchargecomp} can
be easily attained. For instance, one can require that the mesh is constructed in such a way that
for all cells $\calK$ in the triangulation $\calT_h$,
\begin{align*}
  \calK \cap \bar \Omega \neq \emptyset &\Leftrightarrow
  \calK \cap \left (\Omega^\star \setminus \bar \Omega \right)= \emptyset,
\end{align*}
and the polynomial degree of the space $\polQ_h$ is no less than that of $\polW_h$.
Finally, we remark that the nestedness assumption is done merely for convenience.
\end{rem}

The fully discrete problem searches for
\[
  \left\{V_{h\dt}-\bar{V}_{0,\dt},q_{h\dt},\phi_{h\dt},\mu_{h\dt},\bu_{h\dt},p_{h\dt}\right\}
  \subset \polW_h \times \polQ_h^3 \times \polX_h \times \polM_h,
\]
that solve:
\begin{description}
  \item[Initialization] For $n=0$, let $q_h^0$, $\phi_h^0$ and $\bu_h^0$ be suitable approximations of
  the initial charge, phase field and velocity, respectively.
  \item[Time Marching] For $0 \leq n \leq N-1$ we compute
  \[
    (V_h^{n+1},q_h^{n+1},\phi_h^{n+1},\mu_h^{n+1},\bu_h^{n+1},p_h^{n+1})
    \in \polW_h(\bar V_0^{n+1}) \times \polQ_h^3 \times \polX_h \times \polM_h,
  \]
  that solve:
  \begin{equation}
     \scl \vare^\star(\phi_h^{n+1}) \GRAD V_h^{n+1}, \GRAD W_h \scr_{\Omega^\star}
    = \scl q_h^{n+1}, W_h \scr, \quad \forall W_h \in \polW_h,
  \label{eq:potentialdiscrete}
  \end{equation}
  \begin{equation}
    \scl \frac{ \frakd q_h^{n+1} }\dt, r_h \scr
    - \scl q_h^n \bu_h^{n+1}, \GRAD r_h \scr
    + \scl K(\phi_h^n) \GRAD \left( \lambda q_h^{n+1} + V_h^{n+1} \right), \GRAD r_h \scr = 0,
    \quad \forall r_h \in \polQ_h,
  \label{eq:chargediscrete}
  \end{equation}
  \begin{equation}
    \scl \frac{\frakd \phi_h^{n+1}}\dt, \bar\phi_h \scr
    + \scl \bu_h^{n+1} \SCAL \GRAD \phi_h^n, \bar\phi_h \scr
    + \scl M(\phi_h^n) \GRAD \mu_h^{n+1}, \GRAD \bar\phi_h \scr =0, \quad \forall \bar\phi_h \in \polQ_h
  \label{eq:phasediscrete}
  \end{equation}
  \begin{multline}
    \scl \mu_h^{n+1}, \bar\mu_h \scr
    = \frac\gamma\delta \scl \calW'(\phi_h^n) + \calA \frakd\phi_h^{n+1}, \bar\mu_h \scr
    + \gamma\delta \scl \GRAD \phi_h^{n+1}, \GRAD \bar\mu_h \scr \\
    - \frac12 \scl \calE(\phi_h^{n+1}, \phi_h^n) |\GRAD V_h^{n+1}|^2, \bar\mu_h \scr
    + \frac12 \scl \rho'(\phi_h^n) \bu_h^n \SCAL \bu_h^{n+1}, \bar\mu_h \scr \\
    + \alpha \sbl \frac{\frakd \phi_h^{n+1}}\dt + \bu_{h\btau}^{n+1}\psi(\phi_h^n), \bar\mu_h \sbr
    + \gamma \sbl \Theta_{fs}'(\phi_h^n) + \calB\frakd\phi_h^{n+1}, \bar\mu_h \sbr
    \quad \forall \bar\mu_h \in \polQ_h,
  \label{eq:chemdiscrete}
  \end{multline}
  where we introduced
  \begin{equation}
    \calE(\varphi_1, \varphi_2) = \int_0^1  \vare'\left( s \varphi_1 + (1-s)\varphi_2 \right) \diff s,
  \label{eq:defofcalE}
  \end{equation}
  \begin{subequations}
  \label{eq:NSEdiscrete}
  \begin{multline}
  \label{eq:veldiscrete}
    \scl \frac{ \overline{\rho(\phi_h^{n+1})} \bu_h^{n+1} - \rho(\phi_h^n) \bu_h^n }\dt, \bw_h \scr
    + \scl \rho(\phi_h^n) \bu_h^n \ADV \bu_h^{n+1}, \bw_h \scr
    + \frac12 \scl \DIV(\rho(\phi_h^n)\bu_h^n) \bu_h^{n+1}, \bw_h \scr \\
    + \scl \eta(\phi_h^n) \bS(\bu_h^{n+1}), \bS(\bw_h) \scr
    - \scl p_h^{n+1}, \DIV \bw_h \scr
    + \sbl \beta(\phi_h^n) \bu_{h\btau}^{n+1}, \bw_{h\btau} \sbr
    + \alpha \sbl \bu_{h\btau}^{n+1} \psi(\phi_h^n), \bw_{h\btau} \psi(\phi_h^n) \sbr
    \\ =
    \scl \mu_h^{n+1}\GRAD \phi_h^n, \bw_h \scr
    - \scl q_h^n\GRAD( \lambda q_h^{n+1} + V_h^{n+1}), \bw_h \scr
    + \frac12 \scl \rho'(\phi_h^n) \frac{\frakd \phi_h^{n+1}}\dt \bu_h^n, \bw_h \scr \\
    - \alpha \sbl \frac{\frakd \phi_h^{n+1}}\dt, \bw_{h\btau} \psi(\phi_h^n) \sbr
    \quad \forall \bw_h \in \polX_h
  \end{multline}
  \begin{equation}
  \label{eq:presdiscrete}
    \scl \bar{p}_h, \DIV \bu_h^{n+1} \scr = 0, \quad \forall \bar{p}_h \in \polM_h.
  \end{equation}
  \end{subequations}
\end{description}

\begin{rem}[Stabilization parameters]
Notice that, in \eqref{eq:chemdiscrete}, we have introduced two stabilization parameters, namely
$\calA$ and $\calB$. Their purpose is two-fold.  First, they will allow us to
treat the nonlinear terms explicitly while still being able to mantain stability of the scheme, see
Proposition~\ref{prop:denergy} below.
Second, when studying convergence of this problem, the presence of these terms will allow us to obtain
further \emph{a priori} estimates on discrete solutions which, in turn, will help in passing to the limit,
see Theorem~\ref{thm:semidiscreteexists}.
We must mention that, this way of writing nonlinearities is related to the splitting 
of the energy into a convex and concave part 
proposed in \cite{MR2519603}. See also \cite{MR2679727,MR2726065}.
\end{rem}

\begin{rem}[Derivative of the permittivity]
\label{rem:calE}
Notice that \eqref{eq:defofcalE}, \ie the definition of the term $\calE$, is a highly nonlinear function
of its arguments (unless $\vare$ is of a very specific type). As the reader has seen in the derivation of the
energy law (Theorem~\ref{thm:energy}), the treatment of the term involving the derivative of the 
permittivity is subtle. In the fully discrete setting this is additionally complicated by the 
fact that we need
to deal with quantities at different time layers. The reason to write the derivative of the permittivity
in this form is that
\[
  \calE(\varphi_1, \varphi_2 ) = 
  \begin{dcases}
    \frac{ \vare(\varphi_1) - \vare(\varphi_2) }{\varphi_1 - \varphi_2 }, & \varphi_1 \neq \varphi_2, \\
    \vare'(\varphi_1), & \varphi_1 = \varphi_2,
  \end{dcases}
\]
which will allow us to obtain the desired cancellations.
\end{rem}

The following subsections will be devoted to the analysis of problem
\eqref{eq:potentialdiscrete}--\eqref{eq:NSEdiscrete}. For convenience, we
define
\[
  \phi_h^n := \frac{ \frakd \phi_h^n}\dt + \bu_{h\btau}^n \psi(\phi_h^{n-1}).
\]

\subsection{A Priori Estimates and Existence}
\label{sub:stability}

Let us show that, if problem \eqref{eq:potentialdiscrete}--\eqref{eq:NSEdiscrete} has a solution,
it satisfies a discrete energy inequality similar to the one stated in Theorem~\ref{thm:energy}.  
To do this, we first require the following formula, whose proof is straightforward.

\begin{lem}[Summation by parts]
\label{prop:sum_by_parts}
Let $\{ f^n \}^{m-1}_{n=0}$ and $\{ g^n \}^{m-1}_{n=0}$ be sequences and assume $f^{-1} = g^{-1} = 0$.
Then we have
\begin{equation}
\label{eq:sum_by_parts}
  \sum^{m-1}_{n=0} (\frakd g^n) f^n = f^{m-1} g^{m-1} - \sum^{m-2}_{n=0} g^n (\frakd f^{n+1}).
\end{equation}
\end{lem}

\begin{prop}[Discrete stability]
\label{prop:denergy}
Assume that the stabilization parameters $\calA$ and $\calB$ are chosen so that
\begin{equation}
  \calA \geq \frac12 \sup_{\xi \in \Real} \calW''(\xi), \quad
  \calB \geq \frac12 \sup_{\xi \in \Real} \Theta_{fs}''(\xi).
\label{eq:stabparams}
\end{equation}
The solution to \eqref{eq:potentialdiscrete}--\eqref{eq:NSEdiscrete}, if it exists,
satisfies the following a priori estimate
\begin{multline}
\label{eq:denergy}
  \| \bu_{h\dt} \|_{\ell^\infty(\bL^2)}
  + \| \frakd \bu_{h\dt} \|_{\frakh^{1/2}(\bL^2)}
  + \| \bu_{h\dt} \|_{\ell^2(\bV)}
  + \| q_{h\dt}\|_{\ell^\infty(L^2)} + \| \frakd q_{h\dt} \|_{\frakh^{1/2}(L^2)} \\
  + \| \GRAD \phi_{h\dt}\|_{\ell^\infty(\bL^2)} + \| \GRAD \frakd \phi_{h\dt}\|_{\frakh^{1/2}(\bL^2)}
  + \| \calW(\phi_{h\dt})\|_{\ell^\infty(L^1)}
  + \|\GRAD V_{h\dt}\|_{\ell^\infty(\bL^2(\Omega^\star))} 
  + \| \GRAD \frakd V_{h\dt}\|_{\frakh^{1/2}(\bL^2(\Omega^\star))} \\
  + \| \dot \phi_{h\dt} \|_{\ell^2(L^2(\Gamma))}
  + \| \Theta_{fs}(\phi_{h\dt})\|_{\ell^\infty(L^1(\Gamma) )}
  + \| \GRAD \mu_{h\dt} \|_{\ell^2(\bL^2)} + \| \GRAD( \lambda q + V )_{h\dt} \|_{\ell^2(\bL^2)}
  \leq
  c,
\end{multline}
where we have set $\mu_h^0 \equiv 0$, $V_h^0 \equiv 0$ for convenience in writing 
\eqref{eq:denergy}. The constant $c$ depends on the constants $\gamma$, $\delta$, $\alpha$, 
the data of the problem $\bu_h^0$, $\phi_h^0$, $q_h^0$, $\bar V_{0,\dt}$ and $T$, but it does 
not depend on the discretization parameters $h$ or $\dt$, nor the solution of the problem.
\end{prop}
\begin{proof}
We repeat the steps used to prove Theorem~\ref{thm:energy}, \ie set
$\bw_h = 2\dt \bu_h^{n+1}$ in \eqref{eq:veldiscrete},
$\bar p_h = p_h^{n+1}$ in \eqref{eq:presdiscrete},
$\bar\phi_h = 2\dt \mu_h^{n+1}$ in \eqref{eq:phasediscrete},
$\bar\mu_h = -2\frakd \phi_h^{n+1}$ in \eqref{eq:chemdiscrete} and
$r_h = 2\dt (\lambda q_h^{n+1}+V_h^{n+1})$ in \eqref{eq:chargediscrete}.
To treat the time-derivative terms in the discrete momentum equation, we use the identity
\[
    2 \bu_h^{n+1} \SCAL \left( \overline{\rho(\phi_h^{n+1})} \bu_h^{n+1} - \rho(\phi_h^n) \bu_h^n \right) =
    \rho(\phi_h^{n+1}) |\bu_h^{n+1}|^2 - \rho(\phi_h^n) |\bu_h^n|^2
    + \rho(\phi_h^n) |\frakd \bu_h^{n+1}|^2;
\]
see \cite{Guermond20092834,MR2398778}. To obtain control on the explicit terms involving the derivatives
of the Ginzburg-Landau potential $\calW$ and the surface energy density $\Theta_{fs}$,
notice that, for instance,
\[
  \calW(\phi_h^{n+1})- \calW(\phi_h^n) = \calW'(\phi_h^n)\frakd \phi_h^{n+1} + \frac12 \calW''(\xi)
  (\frakd \phi_h^{n+1})^2,
\]
for some $\xi$. Choosing the stabilization constant according to \eqref{eq:stabparams}
(\cf \cite{MR2679727,MR2726065,Shen_Xiaofeng_2010,SalgadoMCL}), we deduce that
\[
  \int_\Omega \left( \calW'(\phi_h^n) + \calA \frakd\phi_h^{n+1} \right)\frakd \phi_h^{n+1}  \geq
  \int_\Omega \frakd \calW(\phi_h^{n+1}).
\]
Adding \eqref{eq:chargediscrete}--\eqref{eq:NSEdiscrete} yields,
\begin{multline}
  \frakd \| \sigma(\phi_h^{n+1})\bu_h^{n+1}\|_{\bL^2}^2
  + \| \sigma(\phi_h^n) \frakd \bu_h^{n+1}\|_{\bL^2}^2
  + \lambda \left( \frakd \|q_h^{n+1}\|_{L^2}^2 + \| \frakd q_h^{n+1}\|_{L^2}^2 \right)
  + \gamma \delta \left( \frakd \|\GRAD \phi_h^{n+1}\|_{\bL^2}^2 \right. \\
  + \left. \| \GRAD \frakd \phi_h^{n+1}\|_{\bL^2}^2 \right)
  + \frac{2\gamma}\delta \int_\Omega \frakd \calW(\phi_h^{n+1})
  + 2\gamma \int_\Gamma \frakd \Theta_{fs}(\phi_h^{n+1})
  + 2\dt \left[
    \left\| \sqrt{\eta(\phi_h^n)}\bS(\bu_h^{n+1}) \right\|_{\bL^2}^2
  \right. \\
    + \left\| \sqrt{\beta(\phi_h^n)} \bu_{h\btau}^{n+1} \right\|_{\bL^2(\Gamma)}^2
    + \left\| \sqrt{M(\phi_h^n)}\GRAD \mu_h^{n+1} \right\|_{\bL^2}^2
    + \left\| \sqrt{K(\phi_h^n)}\GRAD \left( \lambda q_h^{n+1} + V_h^{n+1} \right) \right\|_{\bL^2}^2
  \\
  \left.
    + \alpha \left\| \frac{ \frakd \phi_h^{n+1} }\dt
    + \bu_{h\btau}^{n+1}\psi(\phi_h^n) \right\|_{L^2(\Gamma)}^2
  \right]
  + 2 \scl \frakd q_h^{n+1}, V_h^{n+1} \scr
  \leq
  \scl \calE(\phi_h^{n+1},\phi_h^n) |\GRAD V_h^{n+1}|^2, \frakd \phi_h^{n+1} \scr.
\label{eq:dElaw-sansV}
\end{multline}

Take the difference of \eqref{eq:potentialdiscrete} at time-indices $n+1$ and $n$ to obtain
\[
  \scl \frakd \left( \vare^\star(\phi_h^{n+1}) \GRAD V_h^{n+1} \right),
  \GRAD W_h \scr_{\Omega^\star}
  = \scl \frakd q_h^{n+1}, W_h \scr,
\]
and set $W_h = 2(V_h^{n+1} - \bar V_0^{n+1})$. In view of \eqref{eq:toobasic} we have
\begin{multline*}
  2 \frakd \left( \vare^\star(\phi_h^{n+1}) \GRAD V_h^{n+1} \right) \SCAL
  \GRAD V_h^{n+1}
  =
  \frakd \left( \vare^\star(\phi_h^{n+1}) | \GRAD V_h^{n+1} |^2 \right) \\
  + \vare^\star(\phi_h^n) | \GRAD \frakd V_h^{n+1} |^2
  + \frakd\left( \vare^\star(\phi_h^{n+1}) \right) | \GRAD V_h^{n+1} |^2,
\end{multline*}
whence
\begin{multline}
  \frakd \left\| \sqrt{\vare^\star(\phi_h^{n+1}) }
    \GRAD V_h^{n+1} \right\|_{\bL^2(\Omega^\star)}^2
  + \left\| \sqrt{\vare^\star(\phi_h^n) }
    \GRAD \frakd V_h^{n+1} \right\|_{\bL^2(\Omega^\star)}^2
  + \int_{\Omega^\star} \frakd \vare^\star(\phi_h^{n+1}) |\GRAD V_h^{n+1}|^2 \\
  = 2 \scl \frakd q_h^{n+1}, V_h^{n+1} \scr
  - 2 \scl \frakd q_h^{n+1}, \bar V_0^{n+1} \scr
  + 2 \scl \frakd\left( \vare^\star(\phi_h^{n+1}) \GRAD V_h^{n+1} \right),
     \GRAD \bar V_0^{n+1} \scr_{\Omega^\star}.
\label{eq:dEVolt}
\end{multline}

Add \eqref{eq:dElaw-sansV} and \eqref{eq:dEVolt}. Notice that, since the permittivity is assumed constant on
$\Omega^\star \setminus \bar\Omega$, on the left hand side of the resulting inequality we have
the following term
\[
  \int_\Omega \left( \frakd\vare(\phi_h^{n+1}) - \calE(\phi_h^{n+1},\phi_h^n) \frakd\phi_h^{n+1} \right) 
  |\GRAD V_h^{n+1}|^2 = 0,  
\]
where we used the definition of $\calE$, see \eqref{eq:defofcalE} and Remark~\ref{rem:calE}.
Therefore, we obtain
\begin{multline}
  \frakd \| \sigma(\phi_h^{n+1})\bu_h^{n+1}\|_{\bL^2}^2
  + \| \sigma(\phi_h^n) \frakd \bu_h^{n+1}\|_{\bL^2}^2
  + \lambda \left( \frakd \|q_h^{n+1}\|_{L^2}^2 + \frac12 \| \frakd q_h^{n+1}\|_{L^2}^2 \right) \\
  + \gamma \delta \left( \frakd \|\GRAD \phi_h^{n+1}\|_{\bL^2}^2
  + \| \GRAD \frakd \phi_h^{n+1}\|_{\bL^2}^2 \right)
  + \frac{2\gamma}\delta \int_\Omega \frakd \calW(\phi_h^{n+1}) \\
  + \frakd \left\| \sqrt{\vare^\star(\phi_h^{n+1}) }
     \GRAD V_h^{n+1} \right\|_{\bL^2(\Omega^\star)}^2
  + \left\| \sqrt{\vare^\star(\phi_h^n) }
    \GRAD \frakd V_h^{n+1} \right\|_{\bL^2(\Omega^\star)}^2
  + 2\gamma \int_\Gamma \frakd \Theta_{fs}(\phi_h^{n+1}) \\
  + 2\dt \left[
    \left\| \sqrt{\eta(\phi_h^n)}\bS(\bu_h^{n+1}) \right\|_{\bL^2}^2
    + \| \sqrt{\beta(\phi_h^n)} \bu_{h\btau}^{n+1}\|_{\bL^2(\Gamma)}^2
  \right. \\
  \left.
    + \|\sqrt{M(\phi_h^n)}\GRAD \mu_h^{n+1}\|_{\bL^2}^2
    + \left\| \sqrt{K(\phi_h^n)}\GRAD \left( \lambda q_h^{n+1} + V_h^{n+1} \right) \right\|_{\bL^2}^2
    + \alpha \left\| \frac{ \frakd \phi_h^{n+1} }\dt + \bu_{h\btau}^{n+1}\psi(\phi_h^n) \right\|_{L^2(\Gamma)}^2
  \right] \\
  \leq
  -2 \scl \frakd q_h^{n+1}, \bar V_0^{n+1} \scr
  + 2 \scl \frakd \left( \vare^\star(\phi_h^{n+1}) \GRAD V_h^{n+1} \right),
     \GRAD \bar V_0^{n+1} \scr_{\Omega^\star}.
\label{eq:dElawfinalM1}
\end{multline}

Summing \eqref{eq:dElawfinalM1} for $n=0,...,m-1$, using summation by parts \eqref{eq:sum_by_parts} 
(set $\mu_h^0 \equiv 0$, $V_h^0 \equiv 0$), applying the Cauchy-Schwarz and weighted Young's
inequality, we obtain the result.
\end{proof}

\begin{rem}[Compatibility]
Notice that condition \eqref{eq:voltchargecomp} is needed to obtain the stability estimate,
otherwise $2\dt (\lambda q_h^{n+1}+V_h^{n+1})$ would not be an admissible test function for
\eqref{eq:chargediscrete}.
\end{rem}

The \emph{a priori} estimate \eqref{eq:denergy} allows us to conclude that, for all $h>0$ and $\dt>0$, problem
\eqref{eq:potentialdiscrete}--\eqref{eq:NSEdiscrete} has a solution.

\begin{thm}[Existence]
\label{cor:d-existence}
Assume that the discrete spaces satisfy assumptions \eqref{eq:voltchargecomp} and \eqref{eq:LBB},
the stabilization parameters
$\calA,\ \calB$ are chosen as in Proposition~\ref{prop:denergy}. Then, for all
$h>0$ and $\dt>0$, problem
\eqref{eq:potentialdiscrete}--\eqref{eq:NSEdiscrete} has a solution. Moreover, any solution satisfies
estimate
\eqref{eq:denergy}.
\end{thm}
\begin{proof}
The idea of the proof is to use the ``\emph{method of a priori estimates}'' at each time step.
In other words, for each time step we define a map $\calL^{n+1}$ in such a way that a fixed point 
of $\calL^{n+1}$, if it exists, is a solution of our problem. Then, with the aid of the previously shown
a priori estimates we show that $\calL^{n+1}$ does indeed have a fixed point.

We proceed by induction in the discrete time and assume that we have shown that the problem
has a solution up to $n$.
For each $n = 0,...,N-1$, we define 
\begin{align*}
  \calL^{n+1}: \polW_h(\bar V_0^{n+1}) \times \polQ_h^3 \times \polX_h \times \polM_h
  &\rightarrow
  \polW_h(\bar V_0^{n+1}) \times \polQ_h^3 \times \polX_h \times \polM_h, \\
  (V_h,q_h,\phi_h,\mu_h,\bu_h,p_h)
  &\overset{\calL^{n+1}}{\longmapsto}
  ( \hat V_h, \hat q_h,\hat \phi_h,\hat \mu_h,\hat \bu_h, \hat p_h),
\end{align*}
where the quantities with hats solve
\begin{equation}
   \scl \vare^\star(\phi_h) \GRAD \hat V_h, \GRAD W_h \scr_{\Omega^\star}
  = \scl \hat q_h, W_h \scr, \quad \forall W_h \in \polW_h,
\label{eq:fpvoltage}
\end{equation}
\begin{equation}
  \scl \frac{ \hat q_h - q_h^n }\dt, r_h \scr
  - \scl q_h \bu_h, \GRAD r_h \scr
  + \scl K(\phi_h^n) \GRAD \left( \lambda \hat q_h + \hat V_h \right), \GRAD r_h \scr = 0,
  \quad \forall r_h \in \polQ_h,
\label{eq:fpcharge}
\end{equation}
\begin{equation}
  \scl \frac{\hat \phi_h - \phi_h^n }\dt, \bar\phi_h \scr
  + \scl \bu_h \SCAL \GRAD \phi_h^n, \bar\phi_h \scr
  + \scl M(\phi_h^n) \GRAD \hat \mu_h, \GRAD \bar\phi_h \scr =0, \quad \forall \bar\phi_h \in \polQ_h,
\label{eq:fpphase}
\end{equation}
\begin{multline}
  \scl \hat \mu_h, \bar\mu_h \scr
  = \frac\gamma\delta \scl \calW'(\phi_h^n) + \calA \left(\phi_h - \phi_h^n \right), \bar\mu_h \scr
  + \gamma\delta \scl \GRAD \hat \phi_h, \GRAD \bar\mu_h \scr
  + \frac12 \scl \rho'(\phi_h^n) \bu_h^n \SCAL \bu_h, \bar\mu_h \scr \\
  - \frac12 \scl \calE(\phi_h,\phi_h^n) \GRAD V_h \SCAL \GRAD \hat V_h, \bar\mu_h \scr
  + \alpha \sbl \frac{\hat \phi_h - \phi_h^n }\dt + \bu_{h\btau} \psi(\phi_h^n), \bar\mu_h \sbr \\
  + \gamma \sbl \Theta_{fs}'(\phi_h^n) + \calB \left( \phi_h - \phi_h^n \right), \bar\mu_h \sbr
  \quad \forall \bar\mu_h \in \polQ_h,
\label{eq:fpchem}
\end{multline}
\begin{multline}
  \scl \frac{ \tfrac12 \left( \rho(\phi_h) + \rho(\phi_h^n) \right) \hat \bu_h
    - \rho(\phi_h^n) \bu_h^n }\dt, \bw_h \scr
  + \scl \rho(\phi_h^n) \bu_h^n \ADV \hat \bu_h, \bw_h \scr
  + \frac12 \scl \DIV(\rho(\phi_h^n)\bu_h^n) \hat \bu_h, \bw_h \scr \\
  + \scl \eta(\phi_h^n) \bS(\hat \bu_h ), \bS(\bw_h) \scr
  - \scl \hat p_h, \DIV \bw_h \scr
  + \sbl \beta(\phi_h^n) \hat \bu_{h\btau}, \bw_{h\btau} \sbr
  + \alpha \sbl \bu_{h\btau} \psi(\phi_h^n), \bw_{h\btau} \psi(\phi_h^n) \sbr
  \\ =
  \scl \mu_h \GRAD \phi_h^n, \bw_h \scr
  - \scl q_h \GRAD( \lambda q_h + V_h), \bw_h \scr
  + \frac12 \scl \rho'(\phi_h^n) \frac{\phi_h - \phi_h^n }\dt \bu_h^n, \bw_h \scr \\
  - \alpha \sbl \frac{ \phi_h - \phi_h^n }\dt, \bw_{h\btau} \psi(\phi_h^n) \sbr
  \quad \forall \bw_h \in \polX_h,
\label{eq:fpmom}
\end{multline}
\begin{equation}
  \scl \bar{p}_h, \DIV \hat \bu_h \scr = 0, \quad \forall \bar{p}_h \in \polM_h.
\label{eq:fppres}
\end{equation}
Notice that a fixed point of $\calL^{n+1}$ is precisely a solution of the discrete problem
\eqref{eq:potentialdiscrete}--\eqref{eq:NSEdiscrete}.

To show the existence of a fixed point we must prove that:
\begin{itemize}
  \item The operator $\calL^{n+1}$ is well defined.
  \item If there is a $\calX=(V_h,q_h,\phi_h,\mu_h,\bu_h,p_h)$ for which
  $\calX = \omega \calL^{n+1}\calX$, for some $\omega \in [0,1]$, then
  \begin{equation}
    \| \calX \| \leq M,
  \label{eq:omegafpbdd}
  \end{equation}
  where $M>0$ does not depend on $\calX$ or $\omega$.
\end{itemize}
Then,
an application of the Leray-Schauder theorem \cite{MR1625845,MR816732} will allow us to conclude.
Moreover, since a fixed point of $\calL^{n+1}$ is precisely a solution of our problem,
Proposition~\ref{prop:denergy} gives us the desired stability estimate for this solution.

Let us then proceed to show these two points:

\noindent\underline{The operator $\calL^{n+1}$ is well defined:} Clearly, for any given $\phi_h$,
and $q_h$, the system \eqref{eq:fpvoltage}--\eqref{eq:fpcharge} is positive definite and, thus,
there are unique $\hat V_h$ and $\hat q_h$. Having computed $\hat V_h$ and $\hat q_h$ we then notice that
\eqref{eq:fpmom} and \eqref{eq:fppres} are nothing but a discrete version of a generalized Stokes
problem. Assumption \eqref{eq:LBB} then shows that there is a unique pair $(\hat\bu_h,\hat p_h)$.
To conclude, use $(\hat V_h, \hat q_h, \hat\bu_h, \hat p_h)$ as data in \eqref{eq:fpphase} and
\eqref{eq:fpchem}. The fact that this linear system has a unique solution can then be seen, for instance,
by noticing that the system matrix is positive definite.

\noindent\underline{Bounds on the operator:}
Notice, first of all, that one of the assumptions of the Leray-Schauder theorem is the 
compactness of the operator 
for which we are looking for a fixed point. However, this is trivial since the spaces 
we are working on are finite dimensional.
Let us now show the bounds noticing that, at this stage, we do not need to obtain bounds that are
independent of $h$, $\dt$ or the solution at the previous step. This will be a consequence of 
Proposition~\ref{prop:denergy}. Let us then assume that for some 
$\calX=(V_h,q_h,\phi_h,\mu_h,\bu_h,p_h)$ we have $\calX = \omega \calL^{n+1} \calX$. Notice, first
of all, that if $\omega=0$ then $\calX=0$ and the bound is trivial. If $\omega \in (0,1]$,
the existence of such element can be identified with replacing, in 
\eqref{eq:fpvoltage}--\eqref{eq:fppres}, 
$(\hat V_h, \hat q_h,\hat \phi_h, \hat \mu_h,\hat \bu_h,\hat p_h)$ by
$\omega^{-1}(V_h,q_h,\phi_h,\mu_h,\bu_h,p_h)$. Having done that, set 
$\bw_h = 2\dt u_h$ in \eqref{eq:fpmom},
$r_h = 2\dt(\lambda q_h + V_h)$ in \eqref{eq:fpcharge},
$\bar \phi_h = 2\dt \mu_h$ in \eqref{eq:fpphase} and
$\bar \mu_h = 2(\phi_h - \phi_h^n)$ in \eqref{eq:fpchem}.
Next we observe that, by induction, the equation has a solution at the previous time step, therefore
there are functions that satisfy \eqref{eq:potentialdiscrete} for time $n$. Multiply this 
identity by $\omega$ and subtract it from \eqref{eq:fpvoltage}.
Arguing as in the proof of Proposition~\ref{prop:denergy} we see that
condition \eqref{eq:stabparams} implies that to obtain the desired bound we must 
prove estimates for the terms
\[
  \scl \rho'(\phi_h^n) \bu_h^n \bu_h, \phi_h^n \scr, \quad
  \scl \mu_h, \phi_h^n \scr, \quad
  \scl q_h^n, V_h \scr,  
\]
which are, in a sense, the price we are paying for not being fully impicit. All these terms
are linear $\calX$ and, thus, can be easily bounded by taking into account that we are in finite 
dimensions and that the estimates
need not be uniform in $h$ and $\dt$.
\end{proof}

\section{Numerical Experiments}
\label{sec:NumExp}

In this section we present a series of numerical examples aimed at showing the capabilities of
the model we have proposed and analyzed. The implementation of all the numerical experiments has
been carried out with the help of the \texttt{deal.II} library \cite{BHK07,BHK} and the details
will be presented in \cite{SalgadodealII}.

Let us briefly describe the discretization technique. Its starting point is problem
\eqref{eq:potentialdiscrete}--\eqref{eq:NSEdiscrete}
which, being a nonlinear problem, we linearize with time-lagging of the variables.
Moreover, for the Cahn Hilliard Navier Stokes part we employ the fractional time-stepping technique
developed in \cite{SalgadoMCL}. In other words, at each time step we know
\[
  (V_h^n,q_h^n,\phi_h^n,\mu_h^n,\bu_h^n,p_h^n,\xi_h^n)
  \in \polW_h(\bar V_0^n) \times \polQ_h^3 \times \polX_h \times \polM_h^2,
\]
with $\xi_h^0 := 0$
and, to advance in time, solve the following sequence of discrete and linear problems:
\begin{itemize}
  \item Step 1:
\begin{description}
  \item[Potential] Find $V_h^{n+1} \in \polW_h(\bar V_0^{n+1})$ that solves:
  \[
    \scl \vare^\star(\phi_h^n) \GRAD V_h^{n+1}, \GRAD W_h \scr_{\Omega^\star}
    = \scl q_h^n, W_h \scr, \quad \forall W_h \in \polW_h,
  \]
  \item[Charge] Find $q_h^{n+1} \in \polQ_h$ that solves:
  \[
    \scl \frac{ \frakd q_h^{n+1} }\dt, r_h \scr
    - \scl q_h^n \bu_h^n, \GRAD r_h \scr
    + \scl K(\phi_h^n) \GRAD \left( \lambda q_h^{n+1} + V_h^{n+1} \right), \GRAD r_h \scr = 0,
    \quad \forall r_h \in \polQ_h,
  \]
\end{description}

  \item Step 2:
\begin{description}
  \item[Phase Field and Potential] Find $\phi_h^{n+1},\; \mu_h^{n+1} \in \polQ_h$ that solve:
  \[
    \scl \frac{\frakd \phi_h^{n+1}}\dt, \bar\phi_h \scr
    + \scl \bu_h^n \SCAL \GRAD \phi_h^n, \bar\phi_h \scr
    + \scl M(\phi_h^n) \GRAD \mu_h^{n+1}, \GRAD \bar\phi_h \scr =0, \quad \forall \bar\phi_h \in \polQ_h,
  \]
  \begin{multline*}
    \scl \mu_h^{n+1}, \bar\mu_h \scr
    = \frac\gamma\delta \scl \calW'(\phi_h^n) + \calA \frakd\phi_h^{n+1}, \bar\mu_h \scr
    + \gamma\delta \scl \GRAD \phi_h^{n+1}, \GRAD \bar\mu_h \scr \\
    - \frac12 \scl \vare'(\phi_h^n)  |\GRAD V_h^{n+1}|^2, \bar\mu_h \scr
    + \frac12 \scl \rho'(\phi_h^n) |\bu_h^n|^2, \bar\mu_h \scr \\
    + \alpha \sbl \frac{\frakd \phi_h^{n+1}}\dt + \bu_{h\btau}^n \psi(\phi_h^n), \bar\mu_h \sbr
    + \gamma \sbl \Theta_{fs}'(\phi_h^n) + \calB\frakd\phi_h^{n+1}, \bar\mu_h \sbr
    \quad \forall \bar\mu_h \in \polQ_h,
  \end{multline*}
\end{description}

  \item Step 3:
\begin{description}
  \item[Velocity] Define $p_h^\sharp = p_h^n + \xi_h^n$, then find $\bu_h^{n+1} \in \polX_h$
  such that
  \begin{multline*}
    \scl \frac{ \overline{\rho(\phi_h^{n+1})} \bu_h^{n+1} - \rho(\phi^n) \bu_h^n }\dt, \bw_h \scr
    + \scl \rho(\phi_h^n) \bu_h^n \ADV \bu_h^{n+1}, \bw_h \scr
    + \frac12 \scl \DIV(\rho(\phi_h^n)\bu_h^n) \bu_h^{n+1}, \bw_h \scr \\
    + \scl \eta(\phi_h^n) \bS(\bu_h^{n+1}), \bS(\bw_h) \scr
    - \scl p_h^\sharp, \DIV \bw_h \scr
    + \sbl \beta(\phi_h^n) \bu_{h\btau}^{n+1}, \bw_{h\btau} \sbr
    + \alpha \sbl \bu_{h\btau}^{n+1} \psi(\phi_h^n), \bw_{h\btau} \psi(\phi_h^n) \sbr
    \\ =
    \scl \mu_h^{n+1}\GRAD \phi_h^n, \bw_h \scr
    - \scl q_h^n\GRAD( \lambda q_h^{n+1} + V_h^{n+1}), \bw_h \scr
    + \frac12 \scl \rho'(\phi_h^n) \frac{\frakd \phi_h^{n+1}}\dt \bu_h^n, \bw_h \scr \\
    - \alpha \sbl \frac{\frakd \phi_h^{n+1}}\dt, \bw_{h\btau} \psi(\phi_h^n) \sbr
    \quad \forall \bw_h \in \polX_h.
  \end{multline*}
\end{description}

  \item Step 4:
\begin{description}
  \item[Penalization and Pressure] Finally, $\xi_h^{n+1}$ and $p_h^{n+1}$ are computed via
  \[
    \scl \GRAD \xi_h^{n+1}, \GRAD \bar p_h \scr = - \frac\varrho\dt \scl \DIV \bu_h^{n+1}, \bar p_h\scr,
    \quad \forall \bar p_h \in \polM_h,
  \]
  where $\varrho := \min\{\rho_1, \rho_2 \}$ and
  \[
    p_h^{n+1} = p_h^n + \xi_h^{n+1}.
  \]
\end{description}

\end{itemize}

\begin{rem}[CFL]
A variant of the subscheme used to solve for the Cahn Hilliard Navier Stokes part of our problem was
proposed in \cite{SalgadoMCL} and shown to be unconditionally stable. In that reference, however,
the equations for the phase field and
velocity are coupled via terms of the form $\scl \bu_h^{n+1}\SCAL \GRAD \phi_h^n, \bar \phi_h \scr$.
If we adopt this approach, coupling steps 2 and 3, and assume that the permittivity 
does not depend on the phase,
it seems possible to show that this variant of the scheme described above is stable under,
\[
  \dt \leq c \delta h.
\]
On the other hand, if we work with full time-lagging of the variables, then it is possible to show
that the scheme is stable under the, quite restrictive, assumption that
\[
  \dt \leq c \delta^2 h^2.
\]
To assess how extreme these conditions are one must remember that, in practice, it is necessary to
set $h = \calO(\delta)$. Nevertheless, computations show that these conditions are suboptimal and just a
standard CFL condition is necessary to guarantee stability of the scheme.
\end{rem}

\subsection{Movement of a Droplet}
\label{sub:drop_move}
The first example aims at showing that, indeed, electric actuation can be used to manipulate a
two-fluid system. The fluid occupies the domain $\Omega = (-5,5)\times(0,1)$ and above
and below there are dielectric plates of thickness
$1/2$, so that $\Omega^\star = (-5,5)\times(-1/2,3/2)$. A droplet of a heavier fluid shaped
like half a circle of radius $1/2$ is centered at the origin and initially at rest. To the right
half of lower plate we apply a voltage, so that
\[
  V_0 = V_{00} \chi_D, \qquad D = \left\{ (x,y)\in \Real^2: x \geq 0,\ y = -\frac12 \right\}.
\]

The density ratio between the two fluids is $\rho_1/\rho_2 = 100$, the viscosity ratio
$\eta_1/\eta_2 = 10$ and the surface tension coefficient is $\gamma = 50$. The conductivity ratio
is $K_1/K_2 = 10$ and the permittivity ratio
$\vare_1/\vare_2 = 5$ and $\vare_D/\vare_2 = 100$. We have set the mobility parameter to be constant
$M = 10^{-2}$, and $\alpha = 10^{-3}$. The slip coefficient is taken constant $\beta = 10$, and
the equilibrium contact angle between the two fluids is $\theta_s = 120^\circ$.
The interface thickness is $\delta = 5\cdot10^{-2}$ and the regularization parameter $\lambda = 0.5$.
The applied voltage is $V_{00} = 20$.

The time-step is set constant and $\dt = 10^{-3}$.
The initial mesh consists of $5364$ cells with two different levels of refinement. Away from the two-fluid
interface the local mesh size is about $0.125$ and, near the interface, the local mesh size is about
$0.03125$. As required in \texttt{deal.II}, the degree of nonconformity of the mesh is restricted to 1 \ie
there is only one hanging node per face. Every $10$ time-steps the mesh is coarsened and refined using as,
heuristic, refinement indicator the $\bL^2$-norm of the gradient of the phase field variable $\phi$. The
number of coarsened and refined cells is such that we try to keep the number of cells constant.

The discrete spaces are constructed with finite elements 
with equal polynomial degree in each coordinate direction and
\[
  \deg \polW_h = 1, \quad 
  \deg\polQ_h = 2, \quad
  \deg\polX_h = 2, \quad \deg\polM_h = 1,
\]
that is the lowest order
quadrilateral Taylor-Hood element. No stabilization is added to the momentum conservation equation,
nor the convection diffusion equation used to define the charge density.

\begin{figure}
\label{fig:drop_move}
\includegraphics[scale=0.25]{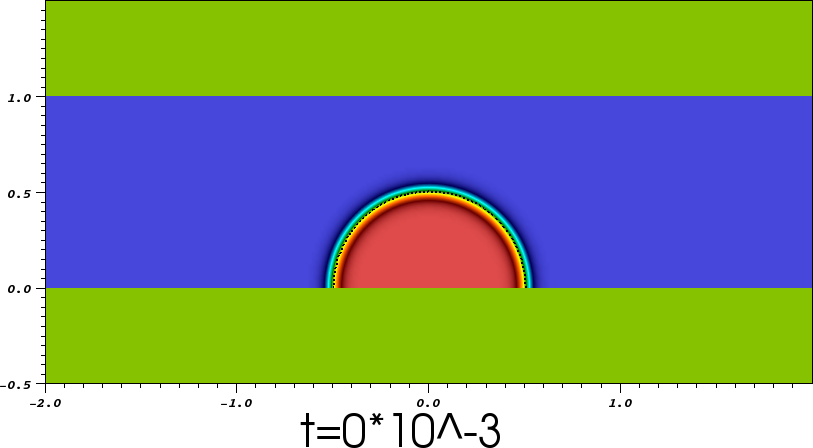}\hfil
\includegraphics[scale=0.25]{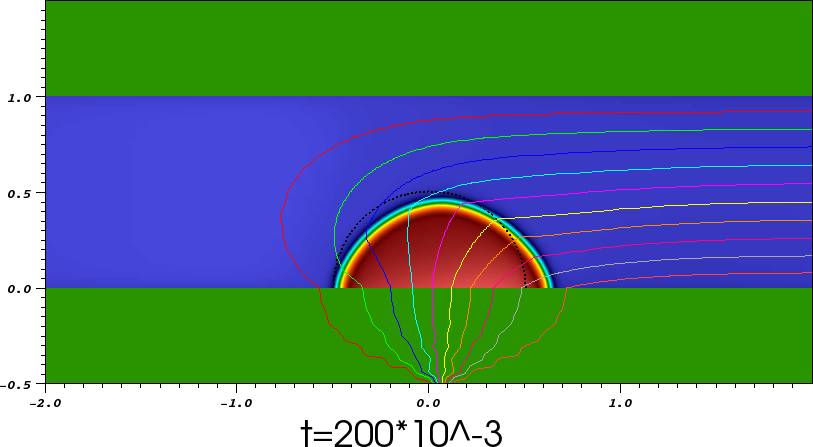} \\
\includegraphics[scale=0.25]{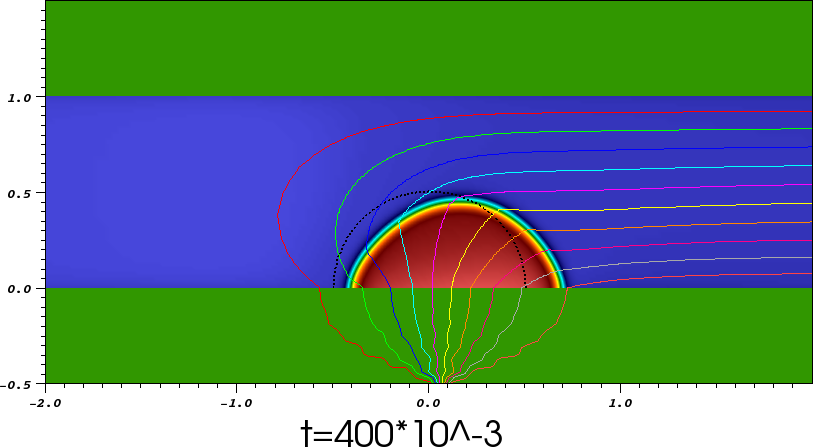}\hfil
\includegraphics[scale=0.25]{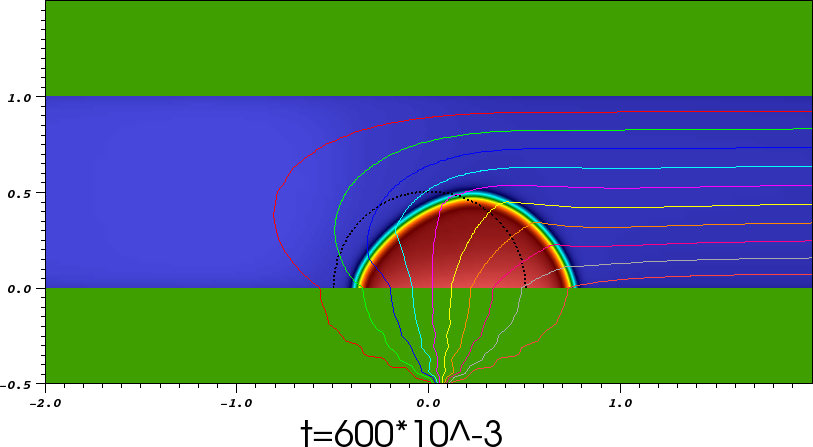} \\
\includegraphics[scale=0.25]{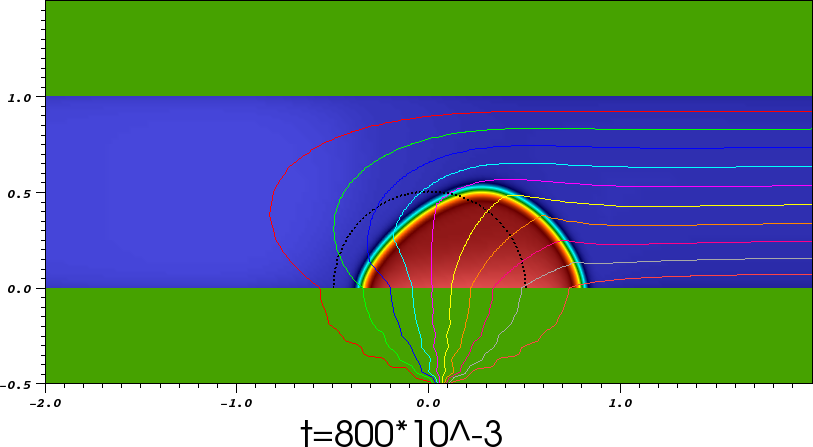}\hfil
\includegraphics[scale=0.25]{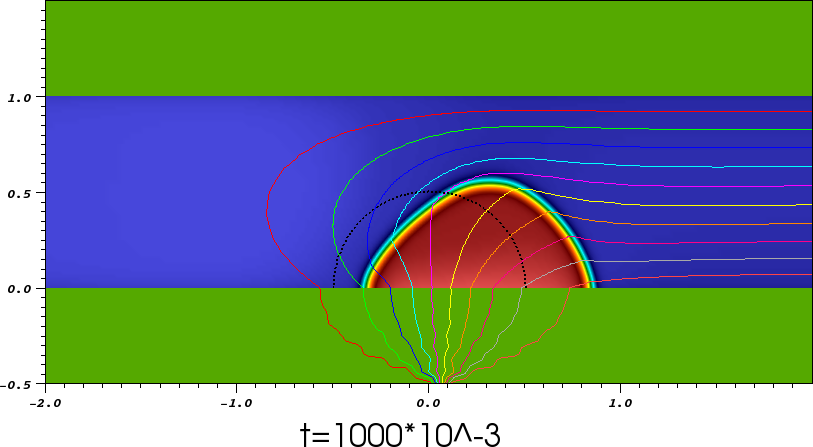} \\
\includegraphics[scale=0.25]{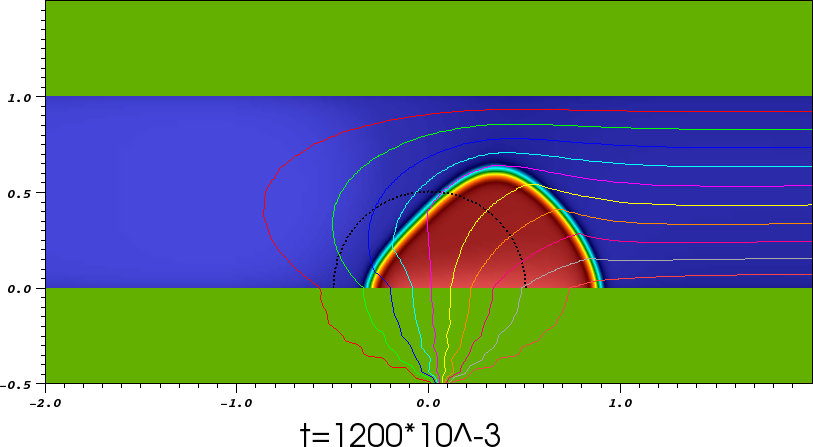}\hfil
\includegraphics[scale=0.25]{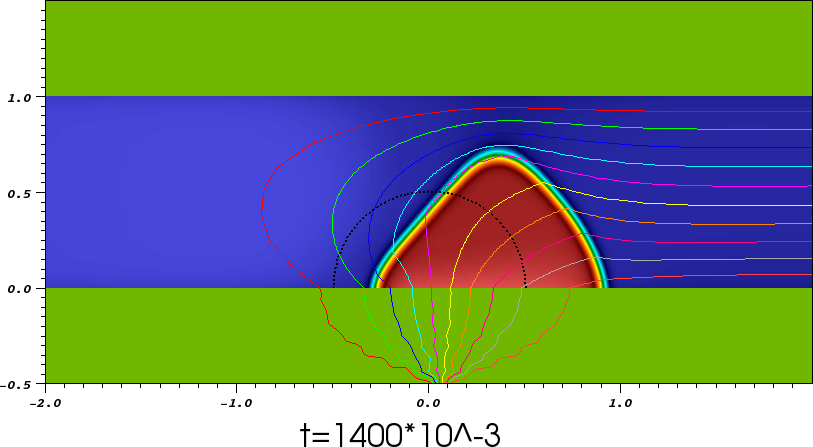}
\caption{Movement of a droplet under the action of an external voltage. The material parameters are
$\rho_1/\rho_2 = 100$, $\eta_1/\eta_2 = 10$, $\gamma = 50$, $K_1/K_2 = 10$, $\vare_1/\vare_2 = 5$,
$\vare_D/\vare_2 = 100$, $M = 10^{-2}$, $\alpha = 10^{-3}$, $\beta = 10$, $\theta_s = 120^\circ$,
$\delta = 5\cdot10^{-2}$, $\lambda = 0.5$ and $V_{00} = 20$. The interface is shown at times
$0$, $0.2$, $0.4$, $0.6$, $0.8$, $1.0$, $1.2$ and $1.4$.
Colored lines are used to represent the iso-values of the voltage. The black dotted line is
the position of the interface at the beginning of the computations.}
\end{figure}

Figure~\ref{fig:drop_move} shows the evolution of the interface.
Notice that, other than adapting the mesh so as to resolve the interfacial layer, no other special
techniques are applied to obtain these results. As expected, the applied voltage creates a 
local modification variation in the value of the surface tension between the two fluids, which in turn
generates a forcing term that drives the droplet.

\subsection{Splitting of a Droplet}
\label{sub:drop_split}
One of the main arguments in favor of diffuse interface models is their ability to handle topological
changes automatically. The purpose of this numerical simulation is to illustrate this by showing
that, using electrowetting, one can split a droplet and, thus, control fluids.
Initially a drop of heavier material occupies
\[
  S_{\rho_2} = \left\{ (x,y) \in \Real^2:\ \frac{x^2}{2.5^2} + \frac{y^2}{0.5^2} = 1 \right\}.
\]
The material parameters are the same as in \S\ref{sub:drop_move}. To be able to split the droplet,
the externally applied voltage is
\[
  D = \left\{ (x,y)\in \Real^2: |x| \geq \frac32,\ y = -\frac12 \right\}.
\]

\begin{figure}
\label{fig:drop_split}
\includegraphics[scale=0.15]{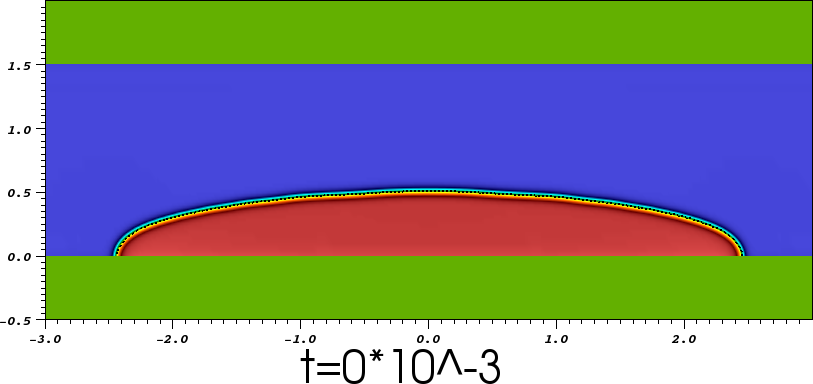}\hfil
\includegraphics[scale=0.15]{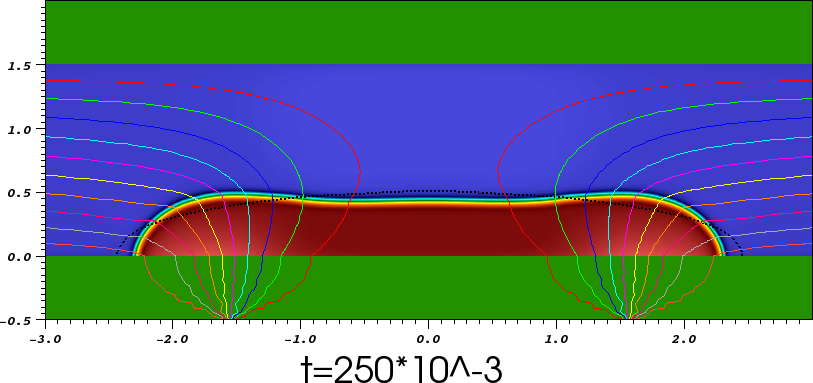}\hfil
\includegraphics[scale=0.15]{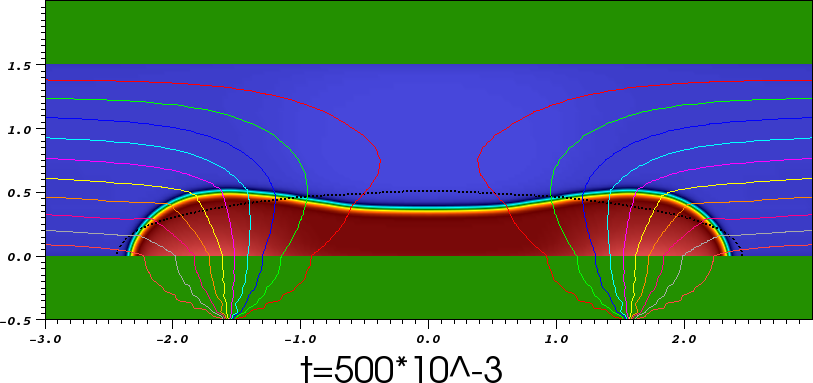}\\
\includegraphics[scale=0.15]{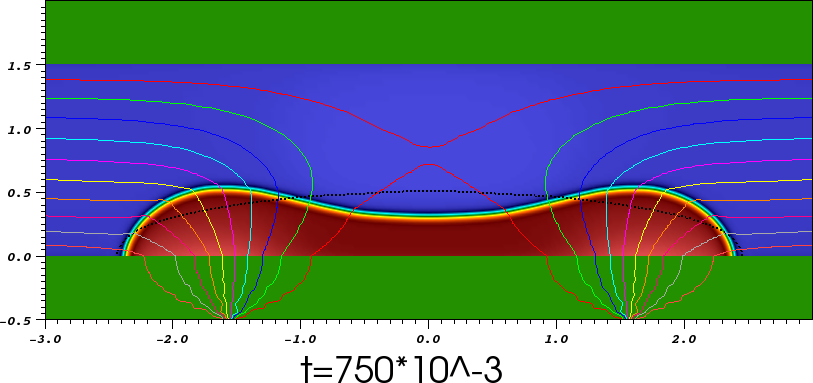}\hfil
\includegraphics[scale=0.15]{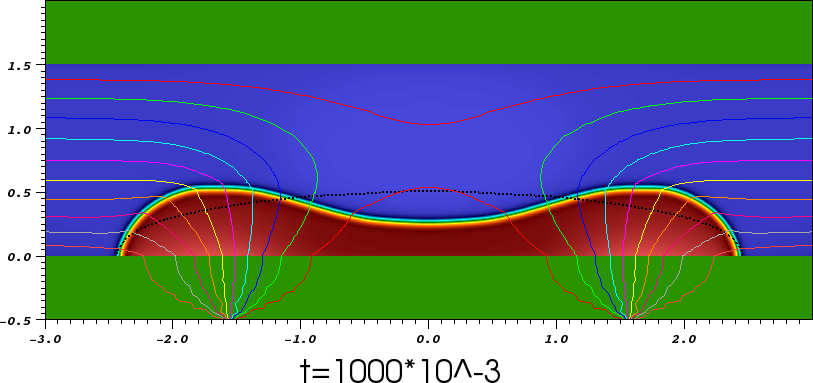}\hfil
\includegraphics[scale=0.15]{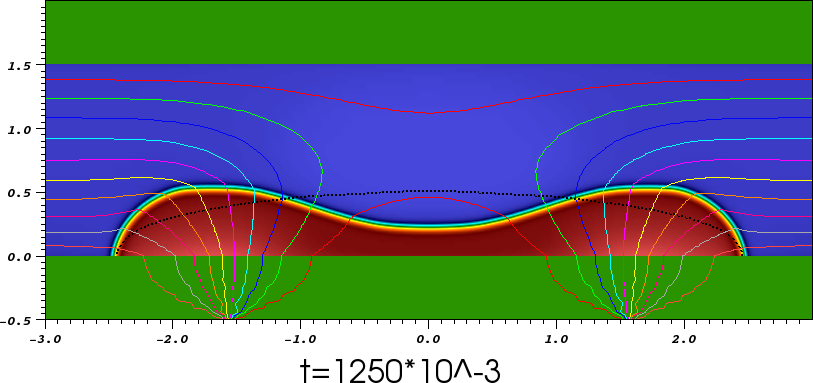}\\
\includegraphics[scale=0.15]{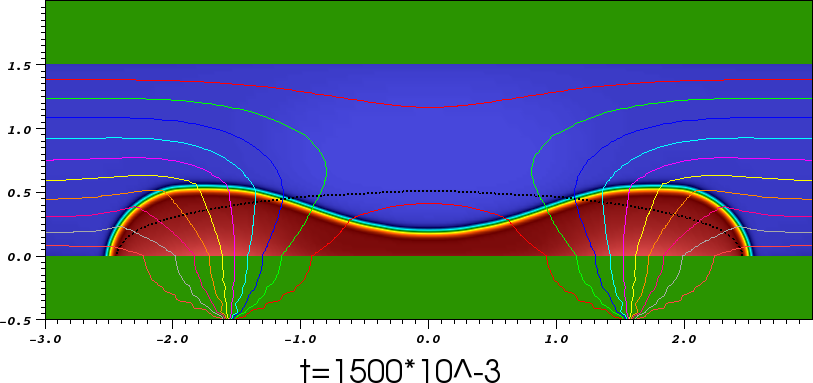}\hfil
\includegraphics[scale=0.15]{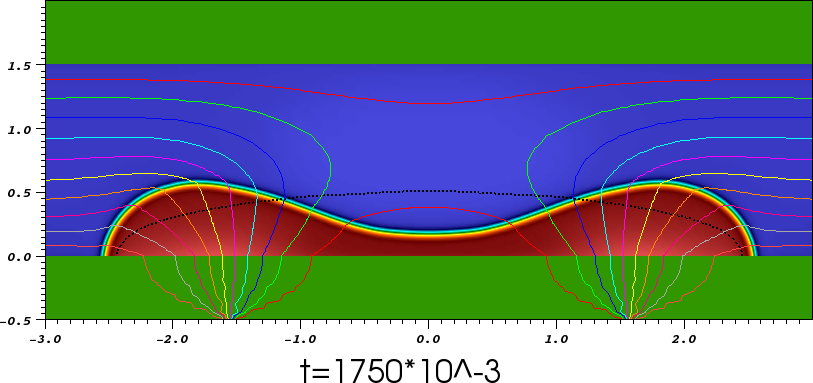}\hfil
\includegraphics[scale=0.15]{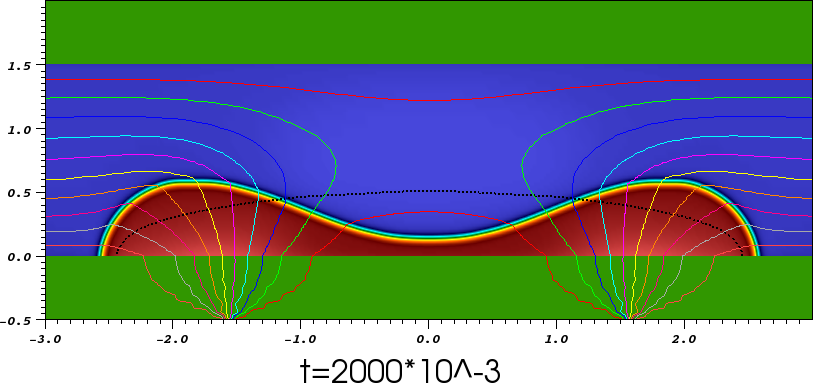}\\
\includegraphics[scale=0.15]{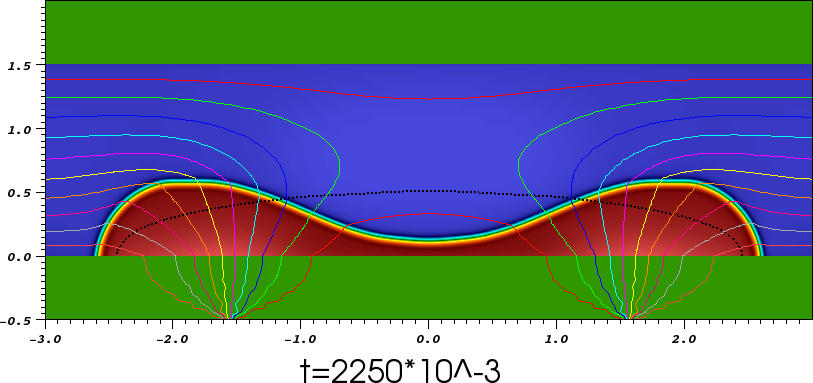}\hfil
\includegraphics[scale=0.15]{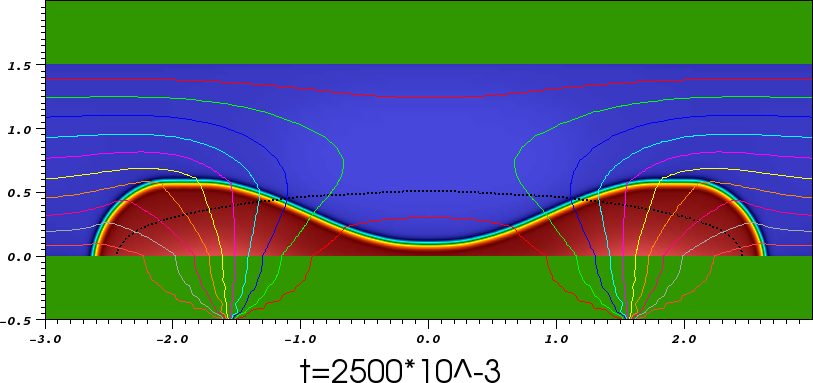}\hfil
\includegraphics[scale=0.15]{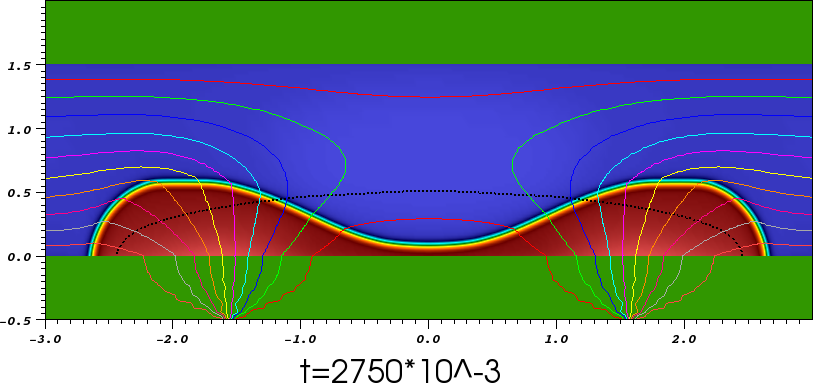}\\
\includegraphics[scale=0.15]{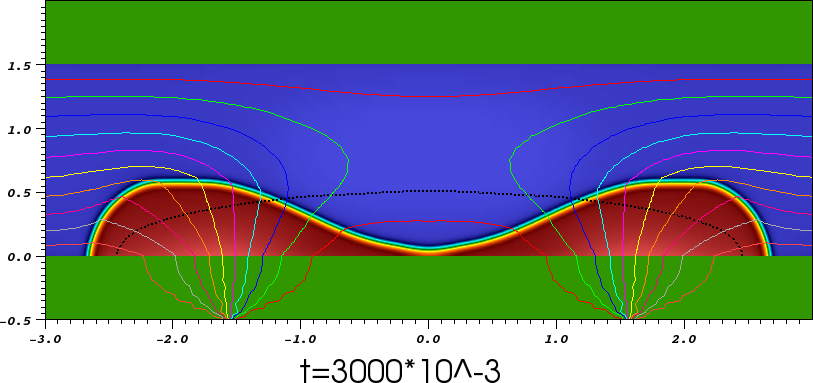}\hfil
\includegraphics[scale=0.15]{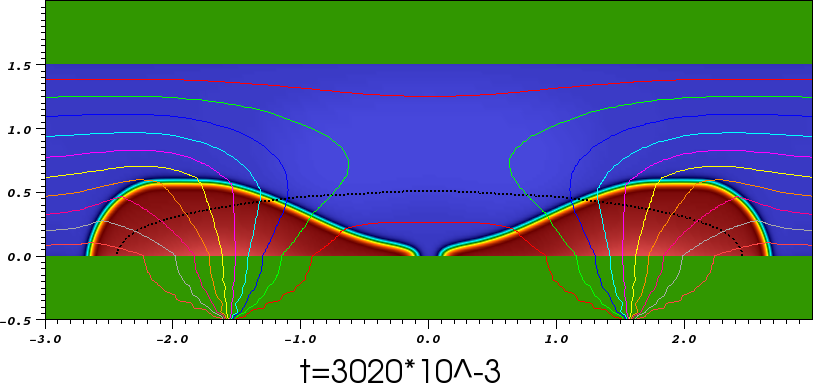}\hfil
\includegraphics[scale=0.15]{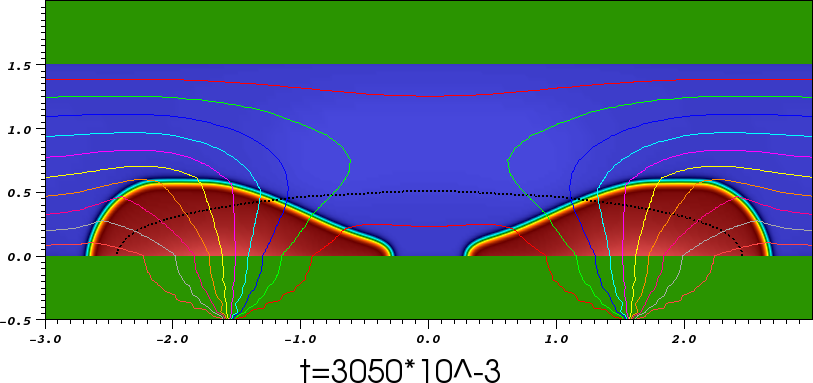}\\
\includegraphics[scale=0.15]{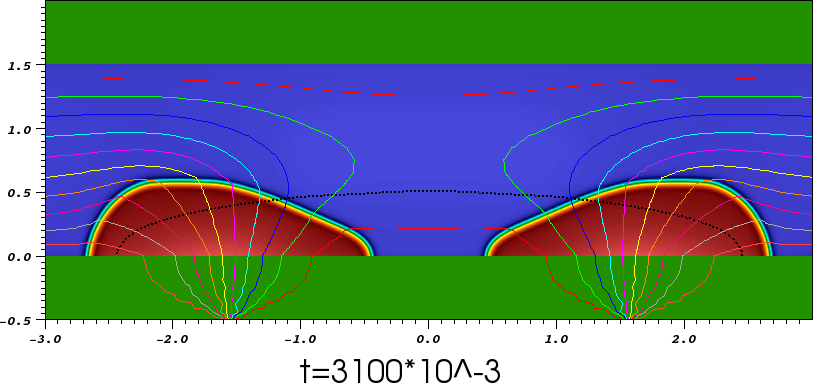}\hfil
\includegraphics[scale=0.15]{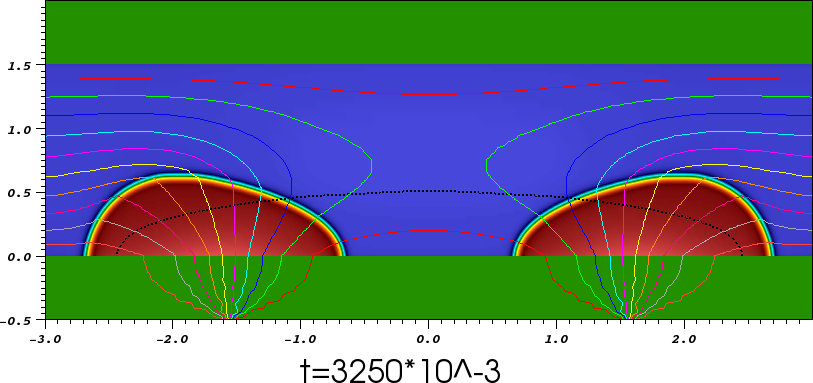}\hfil
\includegraphics[scale=0.15]{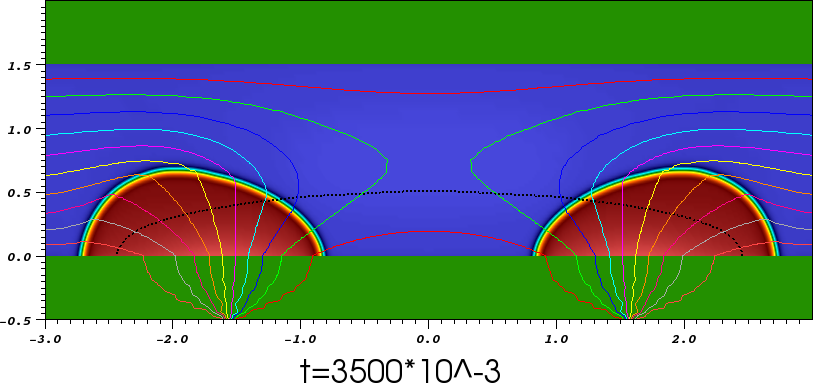}
\caption{Splitting of a droplet under the action of an external voltage. The material parameters are
$\rho_1/\rho_2 = 100$, $\eta_1/\eta_2 = 10$, $\gamma = 50$, $K_1/K_2 = 10$, $\vare_1/\vare_2 = 5$,
$\vare_D/\vare_2 = 100$, $M = 10^{-2}$, $\alpha = 10^{-3}$, $\beta = 10$, $\theta_s = 120^\circ$,
$\delta = 5\cdot10^{-2}$, $\lambda = 0.5$ and $V_{00} = 20$. The interface is shown at times
$0$, $0.25$, $0.5$, $0.75$, $1.0$, $1.25$, $1.5$, $1.75$, $2.0$, $2.25$, $2.5$, $2.75$, $3.0$,
$3.02$, $3.05$, $3.10$, $3.25$ and $3.5$.
Colored lines are used to represent the iso-values of the voltage. The black dotted line is
the position of the interface at the beginning of the computations.}
\end{figure}

Figure~\ref{fig:drop_split} shows the evolution of the system.
Notice that, other than adapting the mesh so as to resolve the interfacial layer, nothing else is done
and the topological change is handled without the necessity to detect it or to adapt the time-step.

\subsection{Merging of Two Droplets}
\label{sub:drop_merge}
To finalize let us show an example illustrating the merging of two droplets of the same material via electric
actuation. The geometrical configuration is the same as in \S\ref{sub:drop_split}. In this case, however, there are initially
two droplets of heavier material, each one of radius $0.5$ and centered at $(-0.7,0)$ and $(0.7,0)$, respectively.
The material parameters are the same as in \S\ref{sub:drop_split}, except the interfacial thickness, which is set to
$\delta=10^{-2}$. We apply an external voltage so that
\[
  D = \left\{ (x,y) \in \Real^2: \ |x| \leq \frac12,\ y = -\frac12 \right\}.
\]

To be able to capture the fine interfacial dynamics that merging possesses, we set the initial level of refinement to $4$, with
$3$ extra refinements near the interface, so that the number of cells is $48,696$ with a local mesh size away of the interface of
about $0.02875$ and near the interface of about $6\cdot10^{-3}$. This amounts to a total of $147,249$ degrees of freedom. The
time-step, again, is set to $\dt = 10^{-3}$.

\begin{figure}
\label{fig:drop_merge}
\includegraphics[scale=0.25]{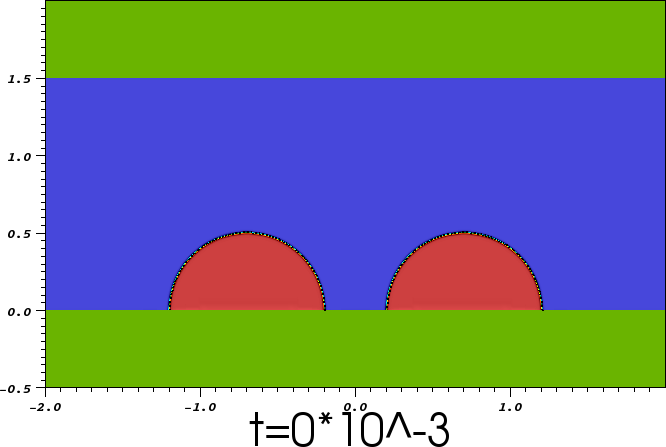}\hfil
\includegraphics[scale=0.25]{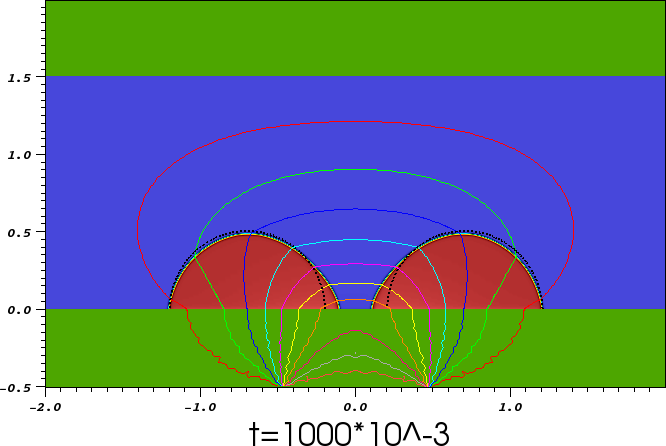}\\
\includegraphics[scale=0.25]{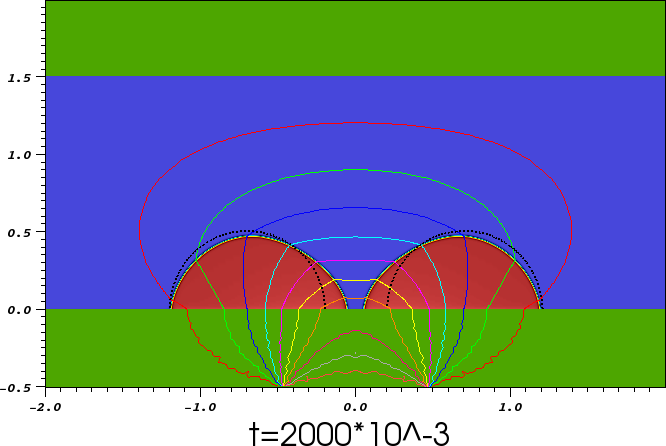}\hfil
\includegraphics[scale=0.25]{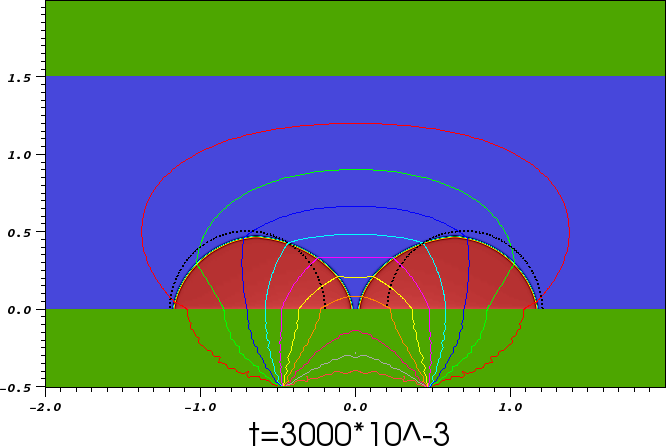}\\
\includegraphics[scale=0.25]{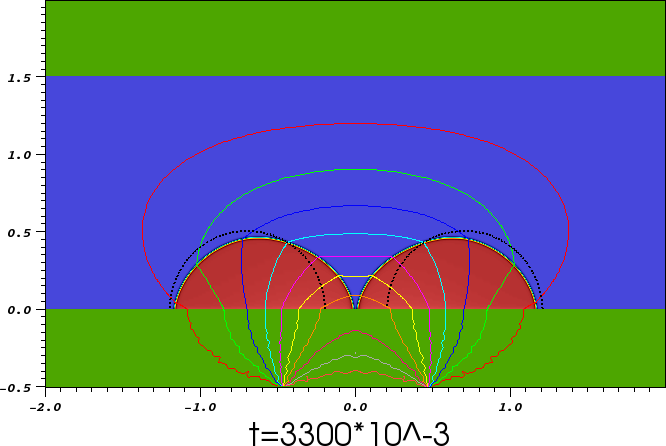}\hfil
\includegraphics[scale=0.25]{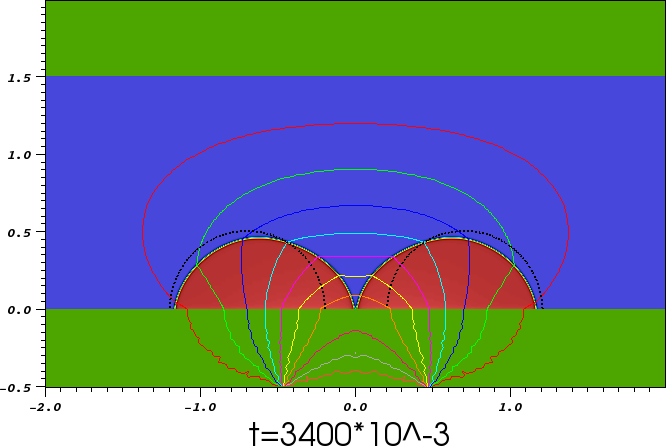}\\
\includegraphics[scale=0.25]{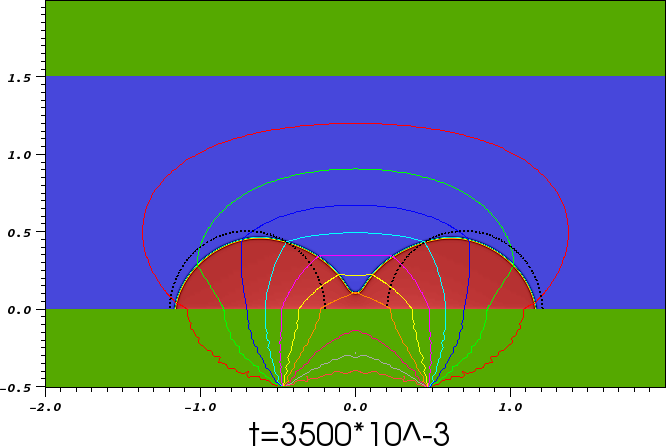}\hfil
\includegraphics[scale=0.25]{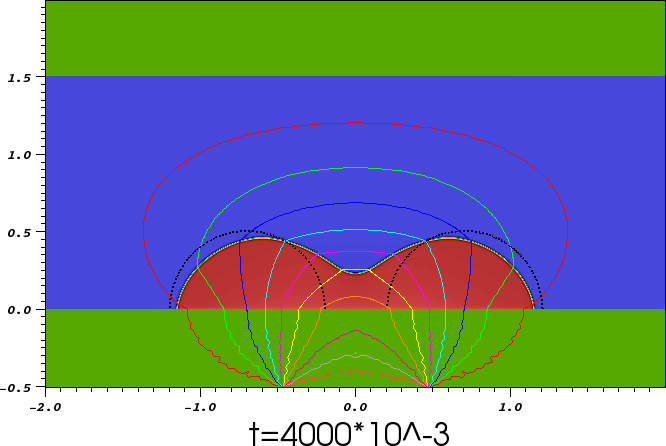}\\
\includegraphics[scale=0.25]{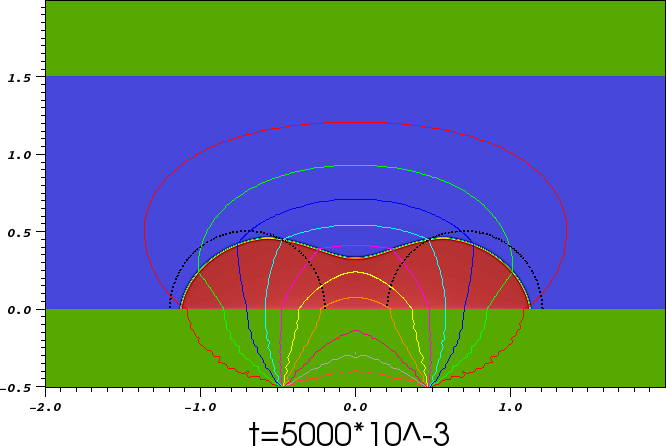}\hfil
\includegraphics[scale=0.25]{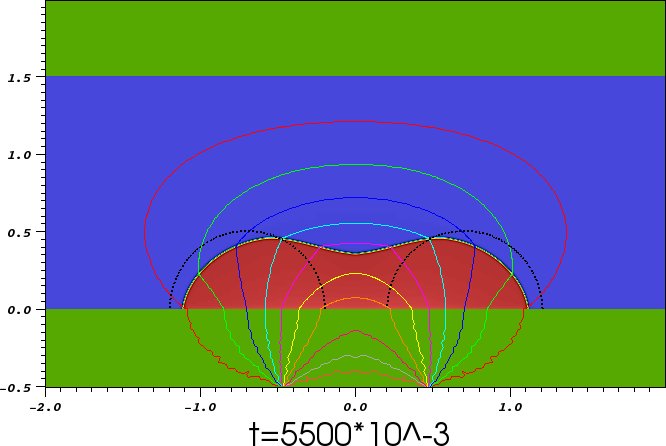}
\caption{Merging of two droplets under the action of an externally applied voltage.
The material parameters are
$\rho_1/\rho_2 = 100$, $\eta_1/\eta_2 = 10$, $\gamma = 50$, $K_1/K_2 = 10$, $\vare_1/\vare_2 = 5$,
$\vare_D/\vare_2 = 100$, $M = 10^{-2}$, $\alpha = 10^{-3}$, $\beta = 10$, $\theta_s = 120^\circ$,
$\delta = 10^{-2}$, $\lambda = 0.5$ and $V_{00} = 20$. The interface is shown at times
$0$, $1$, $2$, $3$, $3.3$, $3.4$, $3.5$, $4$, $5$ and $5.5$.
Colored lines are used to represent the iso-values of the voltage. The black dotted line is
the position of the interface at the beginning of the computations.}
\end{figure}

Figure~\ref{fig:drop_merge} shows the evolution of the two droplets under the action of the voltage.
Again, other than properly resolving the interfacial layer, we did not need to do anything special to handle 
the topological change.

\section{The Semi-Discrete Problem}
\label{sec:semidiscrete}
In \S\ref{sub:stability} we showed that the fully discrete problem always has a solution and that, moreover,
this solution satisfies certain \emph{a priori} estimates. Our purpose here is to pass to the limit
for $h \rightarrow 0$ so as to show that a semi-discrete (that is continuous in space and discrete in time)
version of our electrowetting model always has a solution.

Let us begin by defining the semi-discrete
problem. Given initial data and an external voltage, we find:
\[
  \left\{V_{\dt}-\bar{V}_{0,\dt},q_{\dt},\phi_{\dt},\mu_{\dt},\bu_{\dt},p_{\dt}\right\}
  \subset \Hunstar \times \Hun^3 \times \bV \times \tildeLdeux
\]
that solve:
\begin{description}
  \item[Initialization] For $n=0$, let $q^0$, $\phi^0$ and $\bu^0$ equal
  the initial charge, phase field and velocity, respectively.
  \item[Time Marching] For $0 \leq n \leq N-1$ we compute
  \[
    (V^{n+1},q^{n+1},\phi^{n+1},\mu^{n+1},\bu^{n+1},p^{n+1})
    \in \Hunstar + \bar V_0^{n+1} \times \Hun^3 \times \bV \times \tildeLdeux,
  \]
  that solve:
  \begin{equation}
    \scl \vare^\star(\phi^{n+1}) \GRAD V^{n+1}, \GRAD W \scr_{\Omega^\star}
    = \scl q^{n+1}, W \scr, \quad \forall W \in H^1_0(\Omega^\star),
  \label{eq:potentialsemidiscrete}
  \end{equation}
  \begin{equation}
    \scl \frac{ \frakd q^{n+1} }\dt, r \scr
    - \scl q^n \bu^{n+1}, \GRAD r \scr
    + \scl K(\phi^n) \GRAD \left( \lambda q^{n+1} + V^{n+1} \right), \GRAD r \scr
    = 0, \quad \forall r \in \Hun
  \label{eq:chargesemidiscrete}
  \end{equation}
  \begin{equation}
    \scl \frac{\frakd \phi^{n+1}}\dt, \bar\phi \scr
    + \scl \bu^{n+1} \SCAL \GRAD \phi^n, \bar\phi \scr
    + \scl M(\phi^n) \GRAD \mu^{n+1}, \GRAD \bar\phi \scr
    =0, \quad \forall \bar\phi \in \Hun
  \label{eq:phasesemidiscrete}
  \end{equation}
  \begin{multline}
    \scl \mu^{n+1}, \bar\mu \scr
    = \frac\gamma\delta \scl \calW'(\phi^n) + \calA \frakd\phi^{n+1}, \bar\mu \scr
    + \gamma\delta \scl \GRAD \phi^{n+1}, \GRAD \bar\mu \scr
    + \frac12 \scl \rho'(\phi^n) \bu^n \SCAL \bu^{n+1}, \bar\mu \scr \\
    - \frac12 \scl \calE(\phi^{n+1},\phi^n)|\GRAD V^{n+1}|^2, \bar\mu \scr
    + \alpha \sbl \frac{\frakd \phi^{n+1}}\dt + \bu_\btau^{n+1}\psi(\phi^n), \bar\mu \sbr \\
    + \gamma \sbl \Theta_{fs}'(\phi^n) + \calB\frakd\phi^{n+1}, \bar\mu \sbr
    \quad \forall \bar\mu \in \Hun \cap L^\infty(\Omega),
  \label{eq:chemsemidiscrete}
  \end{multline}
  \begin{subequations}
  \label{eq:NSEsemidiscrete}
  \begin{multline}
  \label{eq:velsemidiscrete}
    \scl \frac{ \overline{\rho(\phi^{n+1})} \bu^{n+1} - \rho(\phi^n) \bu^n }\dt, \bw \scr
    + \scl \rho(\phi^n) \bu^n \ADV \bu^{n+1}
    + \frac12 \DIV(\rho(\phi^n)\bu^n) \bu^{n+1}, \bw \scr \\
    + \scl \eta(\phi^n) \bS(\bu^{n+1}), \bS(\bw) \scr
    - \scl p^{n+1}, \DIV \bw \scr
    + \sbl \beta(\phi^n) \bu_{\btau}^{n+1}, \bw_{\btau} \sbr
    + \alpha \sbl \bu_{\btau}^{n+1} \psi(\phi^n), \bw_{\btau} \psi(\phi^n) \sbr \\
     = \scl \mu^{n+1}\GRAD \phi^n, \bw \scr
    - \scl q^n \GRAD( \lambda q^{n+1} + V^{n+1}), \bw \scr
    + \frac12 \scl \rho'(\phi^n) \frac{\frakd \phi^{n+1}}\dt \bu^n, \bw \scr \\
    - \alpha \sbl \frac{\frakd \phi^{n+1}}\dt, \bw_{\btau} \psi(\phi^n) \sbr,
    \quad \forall \bw \in \bV,
  \end{multline}
  \begin{equation}
  \label{eq:pressemidiscrete}
    \scl \bar{p}, \DIV \bu^{n+1} \scr = 0, \quad \forall \bar{p}\in \tildeLdeux.
  \end{equation}
  \end{subequations}
\end{description}

\begin{rem}[Permittivity]
Notice that, in our definition of solution, the test function for equation \eqref{eq:chemsemidiscrete} 
needs to be bounded. This is necessary to make sense of the term
\[
  \scl \calE(\phi^{n+1},\phi^n) |\GRAD V^{n+1}|^2, \bar\mu \scr,
\]
since $\calE$ is bounded by construction and $V^{n+1} \in \Hun$.
The authors of \cite{MR2511642} used a similar choice of test functions and showed, using different
techniques, existence of a solution for their model of electrowetting in the case when 
the permittivity is phase-dependent.
\end{rem}

Since the solution to the fully discrete problem \eqref{eq:potentialdiscrete}--\eqref{eq:NSEdiscrete}
exists for all values of $h>0$ and satisfies uniform bounds, one expects the sequence of discrete
solutions to converge, in some topology, and that the limit is a solution of problem
\eqref{eq:potentialsemidiscrete}--\eqref{eq:NSEsemidiscrete}. The following
result shows that this is indeed the case.

\begin{thm}[Existence and stability]
\label{thm:semidiscreteexists}
For all $\dt>0$, problem \eqref{eq:potentialsemidiscrete}--\eqref{eq:NSEsemidiscrete} has a solution.
Moreover, this solution satisfies an energy estimate, analogous to \eqref{eq:denergy}, where
the constant $c$ might depend on $\dt$ and the data of the problem, but not on the solution.
\end{thm}
\begin{proof}
Theorem~\ref{cor:d-existence} shows the existence, for every $h>0$, of a solution to the fully
discrete problem \eqref{eq:potentialdiscrete}--\eqref{eq:NSEdiscrete} which, moreover, satisfies
estimate \eqref{eq:denergy}. This estimate implies that, for every $n$, as $h\rightarrow0$:
\begin{itemize}
  \item $\calW(\phi_h^n)$ remains bounded in $L^1(\Omega)$. Since the modified
  Ginzburg-Landau potential is a quadratic function of its argument, this implies that there is
  a subsequence, labeled again $\phi_h^n$, that converges weakly in $\Ldeux$.
  \item $\GRAD \phi_h^n$ remains bounded in $\bL^2$. This, together with the previous
  observation, gives us a subsequence that converges weakly in $\Hun$ and strongly in $\Ldeux$.
  \item The strong $L^2$-convergence of $\phi_h^n$ implies that the convergence is
  almost everywhere and, since all the material functions are assumed continuous, the coefficients
  converge also almost everywhere.
  \item There is a subsequence of $\bu_h^{n+1}$ that converges weakly in $\bV$ and strongly in $\Ldeuxd$.
  \item A subsequence of $V_h^n-\bar V_0^n$ converges weakly in $\Hunstar$ and hence strongly in $L^2(\Omega^\star)$.
  \item There is a subsequence of $q_h^{n+1}$ that converges weakly in $\Ldeux$. Moreover, we know that
  $K(\phi_h^n)\GRAD(\lambda q_h^{n+1} + V_h^{n+1})$ converges weakly. By the a.e.~convergence of the
  coefficients and the $L^2$-weak convergence of $\GRAD V_h^{n+1}$ we conclude that
  $\GRAD q_h^{n+1}$ must converge weakly and, thus, the convergence is weak in $\Hun$ and
  strong in $\Ldeux$.
  \item The quantity
  $ \dot\phi_h^{n+1} = \tfrac{\frakd \phi_h^{n+1}}\dt + \bu_{h\btau}^{n+1} \psi(\phi_h^n)$
  remains
  bounded in $L^2(\Gamma)$, which implies that there is a subsequence of $\dot \phi_h^{n+1}$ that
  converges weakly in $L^2(\Gamma)$.
  \item $\GRAD \mu_h^n$ remains bounded in $\Ldeuxd$. Moreover, setting $\bar \mu_h = 1 $
  in \eqref{eq:chemdiscrete} and the observations given above, imply
  \begin{multline*}
    \left| \scl \mu_h^{n+1}, 1 \scr \right| \leq
    \left|  \frac\gamma\delta \scl \calW'(\phi_h^k) + \calA \frakd \phi_h^{k+1},1\scr  \right. \\
    \left. +\frac12 \scl \rho'(\phi_h^n) \bu_h^n, \bu_h^{n+1} \scr
    + \alpha \sbl \dot \phi_h^{k+1},1\sbr
    + \gamma \sbl \Theta_{fs}'(\phi_h^n) + \calB \frakd \phi_h^{n+1}, 1 \sbr \right|
    \leq c,
  \end{multline*}
  which shows that $\int_\Omega \mu_h^{n+1}$ remains bounded and, thus,
  $\mu_h^n$ remains bounded in $\Hun$ and so there is a subsequence that
  converges weakly in $\Hun$ and strongly in $\Ldeux$.
  \item Finally, we use the compatibility condition \eqref{eq:LBB} and the discrete momentum
  equation \eqref{eq:veldiscrete} to obtain an estimate on the pressure $p_h^{n+1}$,
  \begin{multline*}
    c \| p^{n+1}_h \|_{L^2} \leq
      \frac1\dt \| \rho(\phi_h^n)\|_{L^\infty} \|\frakd \bu_h^{n+1}\|_{\bL^2}
      + \frac1\dt \| \frakd \rho(\phi_h^{n+1})\|_{L^\infty} \| \bu_h^{n+1}\|_{\bL^2}
      + \| \rho(\phi_h^n)\|_{L^\infty} \| \bu_h^n \|_{\bH^1} \| \bu_h^{n+1} \|_{\bH^1} \\
      + \| \rho'(\phi_h^n)\|_{L^\infty}\|\GRAD\phi_h^n\|_{\bL^2}\|\bu_h^n\|_{\bH^1}\|\bu_h^{n+1}\|_{\bH^1}
      + \| \eta(\phi_h^n) \|_{L^\infty} \| \bS(\bu_h^{n+1}) \|_{\bL^2}
      + \| \beta(\phi_h^n) \|_{L^\infty} \| \bu_h^{n+1} \|_{\bV} \\
      + \alpha \| \psi(\phi_h^n) \|_{L^\infty(\Gamma)} \| \dot \phi_h^{n+1} \|_{L^2(\Gamma)}
      + \| \mu_h^{n+1} \|_{H^1} \| \GRAD \phi_h^n \|_{\bL^2}
      + \| q_h^{n+1} \|_{H^1} \|\GRAD (\lambda q_h^{n+1} + V_h^{n+1}) \|_{\bL^2} \\
      + \frac1\dt \| \rho'(\phi_h^n) \|_{L^\infty} \| \frakd \phi_h^{n+1} \|_{L^2} \| \bu_h^n \|_{\bV}
      \leq \frac{c}\dt,
  \end{multline*}
  which, for a fixed and positive $\dt$, implies the existence of a $L^2$-weakly convergent subsequence.
\end{itemize}
Let us denote the limit by
\[
  \left\{V_{\dt}-\bar{V}_{0,\dt},q_{\dt},\phi_{\dt},\mu_{\dt},\bu_{\dt},p_{\dt}\right\}
  \subset \Hunstar \times \Hun^3 \times \bV \times \tildeLdeux.
\]
It remains to show that this limit is a solution of
\eqref{eq:potentialsemidiscrete}--\eqref{eq:NSEsemidiscrete}:

\begin{description}
  \item[Equation \eqref{eq:potentialsemidiscrete}]
  Notice that if we show that, as $h \rightarrow 0$, the sequence $V_h^{n+1}$ converges to
  $V^{n+1}$ strongly in $\Hunstar$, then the a.e. convergence of the coefficients implies
  \begin{equation}
    \scl \vare^\star(\phi_h^{n+1}) \GRAD V_h^{n+1}, \GRAD W \scr_{\Omega^\star}
    \rightarrow
    \scl \vare^\star(\phi^{n+1}) \GRAD V^{n+1}, \GRAD W \scr_{\Omega^\star}.
    \label{eq:potentialconv}
  \end{equation}
  Let us then show the strong convergence by an
  argument similar to that of \cite[pp.~2778]{MR2511642}.
  For any function $V \in \Hunstar,$ we introduce 
  the elliptic projection $\calP_h V \in \polW_h(V)$ as the solution to
  \[
    \scl \GRAD \calP_h V, \GRAD W_h \scr_{\Omega^\star}
    =
    \scl \GRAD V, \GRAD W_h \scr_{\Omega^\star}, \quad \forall W_h \in \polW_h(0).
  \]
  It is well known that $\calP_h V \rightarrow V$ strongly in $\Hunstar$. Given that $\vare$ is uniformly
  bounded,
  \begin{align*}
    c\| \GRAD(V_h^{n+1} - V^{n+1}) \|_{\bL^2}^2 &\leq
    \scl \vare^\star(\phi_h^{n+1}) \GRAD (V_h^{n+1}-V^{n+1}),
      \GRAD  (V_h^{n+1}-V^{n+1}) \scr_{\Omega^\star} \\
    &=
    \scl \vare^\star( \phi_h^{n+1}) \GRAD V_h^{n+1}, 
        \GRAD (\calP_h V^{n+1} - V^{n+1}) \scr_{\Omega^\star} \\
    &+ \scl \vare^\star( \phi_h^{n+1}) \GRAD V_h^{n+1},
        \GRAD ( V_h^{n+1} - \calP_h V^{n+1} ) \scr_{\Omega^\star} \\
    &+  \scl \vare^\star( \phi_h^{n+1}) \GRAD V^{n+1},
        \GRAD ( V^{n+1} - V_h^{n+1}) \scr_{\Omega^\star} \\
    &= I + II +III.
  \end{align*}
  Let us estimate each one of the terms separately. Since the coefficients are bounded and the sequence
  $\GRAD V_h^{n+1}$ is uniformly bounded in $\Ldeuxd$, the strong convergence of $\calP_h V^{n+1}$ shows
  that $I\rightarrow0$. For $II$ we use the equation, namely
  \[
    II=
    \scl \vare^\star( \phi_h^{n+1}) \GRAD V_h^{n+1},
        \GRAD ( V_h^{n+1} - \calP_h V^{n+1} ) \scr_{\Omega^\star}
    = \scl q_h^{n+1},  V_h^{n+1} - \calP_h V^{n+1}  \scr
    \rightarrow 0
  \]
  since $q_h^{n+1}$ converges strongly in $\Ldeux$.
  Finally, notice that the last term can be rewritten as
  \begin{align*}
    III &= \scl \left( \vare^\star( \phi_h^{n+1}) \rrbracket
        - \vare^\star(\phi^{n+1} \right) \GRAD V^{n+1},
        \GRAD ( V^{n+1} - V_h^{n+1}) \scr_{\Omega^\star} \\
    &+ \scl \vare^\star(\phi^{n+1}) \GRAD V^{n+1},
        \GRAD ( V^{n+1} - V_h^{n+1}) \scr_{\Omega^\star} 
  \end{align*}
  The uniform boundedness of $\GRAD V_h^{n+1}$ in $\Ldeuxd$ implies that, for the first term,
  it suffices to show that
  $\left( \vare^\star( \phi_h^{n+1}) 
        - \vare^\star(\phi^{n+1}) \right) \GRAD V^{n+1} \rightarrow 0$ in $\Ldeuxd$,
  which follows from the Lebesgue dominated convergence theorem.
  For the second term, use the weak convergence of $\GRAD V_h^{n+1}$.
  This, together with the strong $L^2$-convergence of $q_h^{n+1}$ implies that the limit solves
  \eqref{eq:potentialsemidiscrete}.

  \item[Equation \eqref{eq:chargesemidiscrete}] The strong $L^2$-convergence of $q_h^{n+1}$ implies that
  $\tfrac1\dt \frakd q_h^{n+1} \rightarrow \tfrac1\dt \frakd q^{n+1}$ strongly in $\Ldeux$.
  Using the compact embeddings $\Hun \Subset L^4(\Omega)$ and $\bV \Subset \bL^4(\Omega)$, we see that
  \[
    \scl q_h^n \bu_h^{n+1}, \GRAD r \scr \rightarrow \scl q^n \bu^{n+1}, \GRAD r \scr,
    \quad \forall r \in \Hun,
  \]
  as $h\rightarrow0$.
  The term $K(\phi_h^n)\GRAD( \lambda q_h^{n+1} + V_h^{n+1})$ can be treated as in \eqref{eq:potentialconv}.
  These observations imply that the limit solves \eqref{eq:chargesemidiscrete}.

  \item[Equation \eqref{eq:phasesemidiscrete}] The strong $\bL^2$-convergence of $\bu_h^{n+1}$, 
  the weak $H^1$-convergence of $\phi_h^{n+1}$ and an
  argument similar to \eqref{eq:potentialconv} imply that the limit solves \eqref{eq:phasesemidiscrete}.

  \item[Equation \eqref{eq:chemsemidiscrete}] The smoothness of $\calW$ and the fact its growth 
  is quadratic imply
  \[
    \left| \scl \calW'(\phi_h^n) - \calW'(\phi^n), \bar \mu \scr\right|
    \leq
    \max_{\varphi} |\calW''(\varphi)| \| \phi_h^n - \phi^n \|_{L^2} \| \bar \mu \|_{L^2}
    \rightarrow
    0.
  \]
  A similar argument and the embedding $\Hun \Subset L^2(\Gamma)$ can be used to show
  convergence of $\Theta_{fs}'(\phi_h^n)$.
  Since $\rho$ is a bounded smooth function,
  \[
    \scl \rho'(\phi_h^n) \bu_h^n \SCAL \bu_h^{n+1}, \bar\mu \scr
    \rightarrow
    \scl \rho'(\phi^n) \bu^n \SCAL \bu^{n+1}, \bar\mu \scr.
  \]
  The strong $\bL^2$-convergence of $\GRAD V_h^{n+1}$ implies that
  \[
    \scl \calE(\phi_h^{n+1},\phi_h^n) |\GRAD V_h^{n+1}|^2, \bar \mu \scr
    \rightarrow
    \scl \calE(\phi^{n+1},\phi^n) |\GRAD V^{n+1}|^2, \bar \mu \scr,
  \]
  where it is necessary to have $\bar\mu \in L^\infty(\Omega)$.
  To conclude that \eqref{eq:chemsemidiscrete} is satisfied by the limit, it is left to show
  that $\dot \phi_h^{n+1}$ converges strongly in $L^2(\Gamma)$.
  We know that $\dot \phi_h^{n+1}$ converges weakly in $L^2(\Gamma)$. On the other hand
  $\tfrac1\dt \frakd \phi_h^{n+1}$ converges strongly in $L^2(\Gamma)$, $\bu_{h\btau}^{n+1}$ converges
  strongly in $\bL^2(\Gamma)$ and $\psi(\phi_h^n)$ converges a.e. in $\Gamma$.

  \item[Equations \eqref{eq:NSEsemidiscrete}] Clearly, \eqref{eq:pressemidiscrete} is satisfied. 
  To show that \eqref{eq:velsemidiscrete} holds, notice that
  \begin{multline*}
    \scl \overline{\rho(\phi_h^{n+1})} \bu_h^{n+1} - \overline{\rho(\phi^{n+1})} \bu^{n+1}, \bw \scr = \\
    \scl \overline{\rho(\phi_h^{n+1})}\left( \bu_h^{n+1} - \bu^{n+1} \right), \bw \scr
    + \scl \left( \overline{\rho(\phi_h^{n+1})} 
    - \overline{\rho(\phi^{n+1})}\right)\bu^{n+1},\bw \scr
    \rightarrow 0.
  \end{multline*}
  Since we assume that $\psi$ is smooth and the
  slip coefficient $\beta$ is smooth and depends only on the phase field, but
  not on the stress (as opposed to \S\ref{sub:pinning}), we can get convergence of the terms
  $ \sbl \beta(\phi_h^n) \bu_{h\btau}^{n+1}, \bw_{\btau} \sbr$ and
  $\sbl \bu_{h\btau}^{n+1} \psi(\phi_h^n), \bw_\btau \psi(\phi_h^n) \sbr$, respectively.
  The advection term
  can be treated using standard arguments and thus we will not give details here. The
  terms
  \[
    \scl \mu_h^{n+1} \GRAD \phi_h^n, \bw \scr, \qquad
    \scl q_h^n \GRAD( \lambda q_h^{n+1} + V_h^{n+1} ), \GRAD \bw \scr,
  \]
  can be treated using arguments similar to the ones given before. The term
  \[
      \scl \rho'(\phi_h^n) \frac{ \frakd\phi_h^{n+1} }\dt \bu_h^n, \bw \scr,
  \]
  can be easily shown to converge since all terms converge strongly. The convergence of
  the term
  \[
    \sbl \frac{\frakd \phi_h^{n+1}}\dt, \bw_{\btau}\psi( \phi_h^n )\sbr,
  \]
  follows again from the compact embedding $\Hun \Subset L^2(\Gamma)$.
  Finally, the convergence of the viscous stress term follows the lines of the proof of
  \eqref{eq:potentialconv}.
\end{description}

To conlcude, we notice that we do not need to reprove estimates similar to \eqref{eq:denergy}. These 
are uniformly valid, in $h$, for all terms in the sequence and, therefore, valid for the limit.
Moreover, if one wanted to obtain an energy estimate by repeating the arguments used to obtain 
Proposition~\ref{prop:denergy} it would be necessary first to obtain uniform $L^\infty$ bounds on
the sequence $\frakd \phi_{h,\dt}$, since one of the steps in the proof requires setting
$\bar \mu_h = 2\frakd \phi_h^{n+1}$.
\end{proof}

\begin{rem}[Limit $\dt \rightarrow 0$]
We are not able to pass to the limit when $\dt \rightarrow 0$ for several reasons. First, the
estimates on the pressure depend on the time-step and getting around this would require finer estimates
on the time derivative of the velocity, this is standard for Navier Stokes.
In addition, the terms
\[
  \scl \rho'(\phi^n) \frac{\frakd \phi^{n+1}}\dt \bu^n, \bw \scr, \qquad
  \scl \frac{ \frakd \rho(\phi^{n+1}) }\dt \bu^{n+1}, \bw \scr,
\]
would require finer estimates on the time derivative of the phase field which we have not been
able to show. It might be possible, however, to circumvent these two restrictions by defining
the weak solution to the continuous problem with an unconstrained formulation for the momentum
equation (\ie solution and test functions in $\bV$) and modifying the Cahn-Hilliard equations
to their ``viscous version'', in other words suitably adding a term of the form $\phi_t$. We
will not pursue this direction.
\end{rem}

\section{Conclusions and Perspectives}
\label{sec:concl}
Some possible directions for future work would be to extend the analysis by passing to the limit as
$\dt \rightarrow 0$, or investigate the phenomenological pinning model more thoroughly.  It would also be
interesting to look at the use of open boundary conditions on $\partial^\star \Omega^\star$, which is more
physically correct for some electrowetting devices.  As far as we know, this is an open area of research in the
context of phase-field methods. Other extensions of the model
could include the transport of surfactant at the liquid-gas interface, though this would make the model more
complicated.  We want to emphasize that our model gives physically reasonable results when modeling actual
electrowetting systems, and so could be used within an optimization framework for improving device design.

Concerning numerics, an important issue that has not been addressed is how to actually solve the 
discretized systems. Even
in the fully uncoupled case, the pressence of the dynamic boundary condition in the Cahn-Hilliard system
(Step 2 in the scheme of section~\ref{sec:NumExp}) makes this problem extremely ill-conditioned and
standard preconditioning techniques (for instance the one in \cite{MR2800707}) inapplicable.

\end{document}